\documentclass[11pt,leqno]{article}
\usepackage{graphicx, amsfonts, amsthm, amsxtra, amssymb, verbatim, makeidx}
\usepackage{subeqnarray, relsize}
\usepackage[mathscr]{euscript}
\usepackage[english]{babel}
\usepackage[fixlanguage]{babelbib}
\usepackage{tikz}
\usepackage[utf8]{inputenc}
\usepackage[english]{babel}
\usepackage{stmaryrd}

\usepackage{wrapfig}
\usepackage{amssymb, amsmath, amsthm}
\usepackage{graphicx}
\usepackage{color}
\usepackage{amssymb}
\usepackage{url}
\usepackage{pdfpages}
\usepackage{fancyhdr}
\usepackage{subfig}
\usepackage{titlesec}
\usepackage{enumerate}
\usepackage{comment}
\usepackage{bigints}
\usepackage{diagbox}
\usepackage{cite}
\usepackage[fixlanguage]{babelbib}
\usepackage[unicode,
            psdextra,
            colorlinks=true,
            linkcolor=blue,
            citecolor=green
            ]{hyperref}
\usepackage[nameinlink,capitalize,noabbrev]{cleveref}

\makeindex
\makeglossary

\theoremstyle{definition}
\newtheorem{dfn}{\bf Definition}[section]
\newtheorem{thm}{Theorem}[section]
\newtheorem{lemma}{Lemma}[section]
\newtheorem{cor}{Corollary}[section]
\newtheorem{prop}{Proposition}[section]

\newtheorem{rmk}{Remark}[section]

\renewenvironment{proof}{{\bfseries \noindent Proof} }{ \qed \\}

\begin{document}

\def\pd{\texttt{pd}}     \def\R{\mathbb{R}}                   
\def\Z{\mathbb{Z}}                   
\def\Q{\mathbb{Q}}                   
\def\C{\mathbb{C}}                   
\def\N{\mathbb{N}}                   
\def\uhp{{\mathbb H}}                
\def\A{\mathbb{A}}
\def\O{\mathcal{O}}

\def\P{\mathbb{P}}
\def\U{\mathcal{U}}
\def\Cl{\text{Cl}}
\def\Pic{\text{Pic}}
\def\Tr{\text{Tr}}
\def\Jac{\text{Jac}}
\def\Hess{\text{Hess}}
\def\Gal{\text{Gal}}
\def\Gr{\text{Gr}}
\def\CH{\text{CH}}
\def\HL{\text{HL}}
\def\codim{\text{codim }}
\def\Hip{\mathbb{H}}
\def\res{\text{res}}
\def\prim{\text{prim}}
\def\dR{\text{dR}}

\def\ker{{\rm ker}}              
\def\GL{{\rm GL}}                
\def\ker{{\rm ker}}              
\def\coker{{\rm coker}}          
\def\im{{\rm Im}}               
\def\coim{{\rm Coim}}            

\def\End{{\rm End}}              
\def\rank{{\rm rank}}                
\def\gcd{{\rm gcd}}                  

\begin{center}
{\LARGE\bf Toric differential forms and periods of complete intersections
}
\footnote{ 
Math. classification: 14C25, 14C30, 14M25, 13H10.
}
\\
\vspace{.25in} {\large {\sc Roberto Villaflor Loyola }}\footnote{Pontificia Universidad Católica de Chile, Facultad de Matemáticas, Campus San Joaquín, Avenida Vicuña Mackenna 4860, Santiago, Chile,
{\tt roberto.villaflor@mat.uc.cl}}
\end{center}


\begin{abstract}
Let $n$ be an even natural number. We compute the periods of any $\frac{n}{2}$-dimensional complete intersection algebraic cycle inside an $n$-dimensional non-degenerated intersection of a projective simplicial toric variety. Using this information we determine the cycle class of such algebraic cycles. As part of the proof we develop a toric generalization of a classical theorem of Macaulay about complete intersection Artin Gorenstein rings, and we generalize an algebraic cup formula for residue forms due to Carlson and Griffiths to the toric setting.
\end{abstract}

\section{Introduction}
\label{intro}
The computation of periods of algebraic varieties is an old problem that can be traced back to the origins of complex algebraic geometry in the works of Abel, Jacobi and Riemann. It was the guiding principle for Picard and Lefschetz in the study of the topology of algebraic varieties, and led Hodge to the discovery of Hodge structures and the so called Hodge conjecture. One way to think about this conjecture is that one can determine from the periods of an homology cycle whether it is supported in an algebraic subvariety or not. The period information is encoded, by Poincar\'e duality, in its cohomology class, and so the Hodge conjecture predicts that the image of the cycle class map corresponds exactly to the rational classes in the middle part of the Hodge decomposition, i.e. to the Hodge cycles. Even when the Hodge conjecture is satisfied, to determine the cycle class of a given algebraic cycle might be hard, e.g. to know the periods of a curve in a surface is harder than knowing all the intersection numbers of the curve with all other curves in the surface. Another example is the computation of periods of top forms, which corresponds to compute Grothendieck residues. To consider this problem explicitly, even for the case of complete simplicial toric varieties is a hard problem \cite{cox1996toric, cattani1997residues, cattani1998residues}. On the other hand, the information encoded by the periods can be used to study moduli problems via the use of infinitesimal variations of Hodge structure \cite{IVHS}. Two classical instances of this are the Torelli problem \cite{CG72, carlson1980infinitesimal, donagi1983generic, voisin1986theoreme, cox1990polynomial, looijenga2009period, voisin2022schiffer} and Noether-Lefschetz theory \cite{CHM88,voisin89, green1989, Otwinowska2003, Maclean05, GMCD-NL}. A survey about how to use periods of algebraic cycles to study components of the Noether-Lefschetz locus and the so called variational Hodge conjecture can be found in \cite[Chapter 18]{ho13}. Probably one of the first results about the values of periods of algebraic cycles is \cite{dmos} where is shown that the periods of an algebraic cycle belong to the algebraic closure of the field of definition of the ambient variety. Explicit computations of periods of linear cycles inside Fermat varieties were obtained in \cite{MV}. A computational approach was developed by Sert\"oz in \cite{emre}. In \cite{villaflor2021periods} was obtained an explicit formula for computing periods of complete intersection algebraic cycles inside smooth hypersurfaces of the projective space. Some applications of these periods computations to Noether-Lefschetz and Hodge loci are \cite{movasati2021special, movasati2021reconstructing, villa2020small, DuqueVillaflor2023fake}. In this work we generalize previous results to compute periods of complete intersection algebraic cycles inside non-degenerated intersections of projective simplicial toric varieties. 

Let $\P_\Sigma$ be an $m$-dimensional projective simplicial toric variety defined by the fan $\Sigma$. Let us denote by $S(\Sigma)=\C[x_1,\ldots,x_r]$ its corresponding Cox ring, with $r=\#\Sigma(1)$, graded by its divisor class group $\Cl(\P_\Sigma)$. We denote by $\beta_0\in \Cl(\P_\Sigma)$ its anti-canonical class. Let $X\subseteq \P_\Sigma$ be a non-degenerated intersection (see \cref{defnondegint}) of ample Cartier divisors defined as the zero locus of $f_i\in S(\Sigma)$ with $\deg(f_i)=\beta_i$, for each $i=1\ldots,s$. Assume that $n:=\dim(X)=m-s$ is an even number. The main result in this article is the computation of the cycle class of all complete intersection (not necessarily quasi-smooth) algebraic cycles $Z\in \CH^\frac{n}{2}(X)$ in terms of its supporting equations
$$
Z=\{ g_1=g_2=\cdots=g_{\frac{m+s}{2}}=0\}\subseteq\P_\Sigma.
$$
In order to describe the primitive cohomology classes of $X$ we use the Cayley trick (see \cref{rmkCayleytrick}) to get the projective simplicial toric variety $\P_{\widetilde{\Sigma}}:=\P(\mathcal{E})$ where $\mathcal{E}:=L_1\oplus\cdots\oplus L_s$ and $[L_i]=\alpha_i\in\Cl(\P_\Sigma)$. We denote by $\mathbb{T}(\widetilde\Sigma)$ its open dense torus, by $M:=\text{Hom}(\mathbb{T}(\widetilde\Sigma),\C^\times)$ and $N:=\text{Hom}(\C^\times,\mathbb{T}(\widetilde\Sigma))\simeq\Z^{m+s-1}$, by $S(\widetilde{\Sigma})=\C[x_1,\ldots,x_r,x_{r+1},\ldots,x_{r+s}]$ its Cox ring and by $$F:=x_{r+1}f_1+\cdots+x_{r+s}f_s\in S(\widetilde{\Sigma})$$ of degree $\beta:=\deg(F)$ (which is also Cartier and ample). Each variable $x_i$ of the Cox ring corresponds to a $\mathbb{T}(\widetilde\Sigma)$-invariant divisor $D_i$ of $\P_{\widetilde\Sigma}$, which in turn is in correspondence with a ray in $\widetilde\Sigma(1)$ generated by a primitive element denoted $\rho_i$. We write $\beta=\sum_{i=1}^{r+s}b_i\cdot D_i$. The Cayley trick gives us an isomorphism
$$
P\in R(F)^{(\frac{m+s}{2})\beta-\beta_0}\xrightarrow{\sim}\omega_P\in H^{\frac{n}{2},\frac{n}{2}}(X)_\prim
$$ 
where $R(F):=S(\widetilde{\Sigma})/J(F)$ is the Jacobian ring, and $J(F):=\langle \frac{\partial F}{\partial x_1},\ldots,\frac{\partial F}{\partial x_r}, f_1,\ldots,f_s\rangle$ (see \S \ref{sec5} for details). Our main result generalizes the main result of \cite[Theorem 1.1]{villaflor2021periods} and describes the cycle class of 
$$
Z\in\CH^\frac{n}{2}(X)\mapsto [Z]\in H^{\frac{n}{2},\frac{n}{2}}(X)\cap H^n(X,\mathbb{Z})
$$
as follows.
\begin{thm}
\label{mainthm}
Let $\alpha_i:=\deg(g_i)\in\Cl(\P_\Sigma)$ be ample Cartier divisors. Then
$$
[Z]=[\alpha]|_X+\omega_{P_Z},
$$
where $\alpha\in \CH^{\frac{n}{2}}(\P_{\Sigma})_\Q$ is the unique cycle such that $$\alpha\cdot \beta_1\cdots\beta_s=\alpha_1\cdots\alpha_{\frac{m+s}{2}},$$ and $P_Z\in R(F)^{(\frac{m+s}{2})\beta-\beta_0}$ can be computed as follows: Let $h_1,\ldots,h_\frac{m+s}{2}\in S(\widetilde{\Sigma})$ be such that 
$$
F=g_1\cdot h_1+g_2\cdot h_2+\cdots+g_\frac{d+s}{2}\cdot h_\frac{d+s}{2}.
$$
Fix any set $\rho_{i_1},\ldots,\rho_{i_{m+s-1}}$ of linearly independent primitive generators of the rays of $\widetilde\Sigma$. For every $\ell\in\{1,2,\ldots,r+s\}\setminus\{i_1,\ldots,i_{m+s-1}\}$ write $\rho_\ell=\sum_{j=1}^{m+s-1}a_{i_j,\ell}\cdot \rho_{i_j}$ and define $k_\ell:=b_\ell-\sum_{j=1}^{m+s-1}b_{i_j}\cdot a_{i_j,\ell}$. Then for every $\ell$ such that $k_\ell\neq 0$ we have
$$
P_Z=\frac{{sgn(I)}\cdot \frac{m+s-2}{2}!\cdot \det\begin{pmatrix}
\frac{\partial g_1}{\partial x_{\ell}} & \frac{\partial h_1}{\partial x_{\ell}} & \cdots & \frac{\partial g_\frac{m+s}{2}}{\partial x_{\ell}} & \frac{\partial h_\frac{m+s}{2}}{\partial x_{\ell}} \\ \frac{\partial g_1}{\partial x_{i_1}} & \frac{\partial h_1}{\partial x_{i_1}} & \cdots & \frac{\partial g_\frac{m+s}{2}}{\partial x_{i_1}} & \frac{\partial h_\frac{m+s}{2}}{\partial x_{i_1}} \\
\vdots & \vdots &  & \vdots & \vdots
\\ \frac{\partial g_1}{\partial x_{i_{m+s-1}}} & \frac{\partial h_1}{\partial x_{i_{m+s-1}}} & \cdots & \frac{\partial g_\frac{m+s}{2}}{\partial x_{i_{m+s-1}}} & \frac{\partial h_\frac{m+s}{2}}{\partial x_{i_{m+s-1}}} 
\end{pmatrix}}{k_\ell\cdot \det(\rho_I)^2\widehat{x_{\ell,I}}},
$$
where $sgn(I)\in\{\pm1\}$ is the sign of the permutation which orders $I=(i_1,\ldots,i_{m+s-1})$ increasingly, $\rho_I$ is the matrix whose columns are $\rho_{i_1}, \ldots, \rho_{i_{m+s-1}}\in\Z^{m+s-1}\simeq N$, and $\widehat{x_{\ell,I}}:=\prod_{i\notin \{\ell\}\cup I}x_i$.
\end{thm}

The main application of the above result is the computation of all non-trivial periods of such algebraic cycles. By non-trivial periods we mean the periods of non-trivial primitive cohomology classes. The classes which vanish in primitive cohomology are dual to homology cycles obtained from the ambient $\P_\Sigma$ by intersection with $X$, and so computing their periods reduces to compute the intersection number of $Z$ with some algebraic cycle in the Chow ring of $\P_\Sigma$. Our second result is as follows.

\begin{thm}
\label{thm2}
In the same context of \cref{mainthm}, for every $P\in R(F)^{(\frac{m+s}{2})\beta-\beta_0}$
$$
\frac{1}{(2\pi i)^\frac{n}{2}}\int_Z\omega_P=- c\cdot \frac{(m+s-1)!}{(\frac{m+s-2}{2}!)^2}\text{Vol}(\Delta)
$$
where $\text{Vol}(\Delta)$ is the normalized volume of the convex polyhedron $\Delta\subseteq M_\R$ associated to $\beta$, and $c\in\C$ is the unique number such that
$$
P\cdot P_Z\cdot x_1\cdots x_{r+s}\equiv c\cdot \Hess_{\widetilde{\Sigma}}(F) \ \ \ \text{(mod }J_0(F))
$$
where $J_0(F)=\langle x_1\frac{\partial F}{\partial x_1},\ldots,x_r\frac{\partial F}{\partial x_r},x_{r+1}f_1,\ldots,x_{r+s}f_s\rangle\subseteq S(\widetilde{\Sigma})$ and $\Hess_{\widetilde{\Sigma}}(F)$ is the toric Hessian of $F$ (see \cref{deftorichess}).
\end{thm}

The above results are equivalent, and depend on toric generalizations of classical results which are developed in the article. One of them is due to Macaulay on Artinian Gorenstein algebras \cite{mac16}, the others are a theorem of residues along hypersurfaces due to Carlson and Griffiths \cite{carlson1980infinitesimal} and the computation of their wedge product. Besides, it is also necessary to use the theory of toric residues developed by Cattani, Cox and Dickenstein \cite{cox1996toric,cattani1997residues}. The main difficulty for translating results of Artinian Gorenstein algebras to the toric setup is that these algebras are not Artinian if the Picard number of the ambient toric variety is bigger than 1. For this reason, the Gorenstein duality only reflects in a proper subset of the degrees indexed by the class group, and so we call such classes \textit{regular classes} (see \cref{defregular}). We generalize Macaulay theorem as follows.

\begin{thm}
\label{thm3}
Let $\P_\Sigma$ be an $m$-dimensional projective simplicial toric variety. Consider $f_0,\ldots,f_m\in S(\Sigma)$ such that $\deg(f_i)=\alpha_i$ is Cartier and ample, and $V(f_0,\ldots,f_m)=\varnothing\subseteq\P_\Sigma$. Let $I:=\langle f_0,\ldots,f_m\rangle$, $R:=S(\Sigma)/I$ and $N:=(\sum_{i=0}^m\alpha_i)-\beta_0$, where $\beta_0$ is the anti-canonical class. Then $\dim_\C R^N=1$, and every Cartier class $\beta\in\Cl(\P_\Sigma)$ such that 
$$
H^q\left(\mathcal{O}_{\P_\Sigma}\left(\beta-\sum_{j=1}^{q+1}\alpha_{i_j}\right)\right)=0
$$
for all $q=1,\ldots, m-1$ and $0\le i_1<\cdots<i_{q+1}\le m$ is regular for $(I,N)$, i.e. the multiplication map $R^\beta\times R^{N-\beta}\rightarrow R^N\simeq \C$ induces an injection $R^\beta\hookrightarrow (R^{N-\beta})^\vee$.
\end{thm}

The equivalence between \cref{mainthm} and \cref{thm2} depends on the following computation of the trace of the cup product of two residue forms along an ample quasi-smooth hypersurface. This generalizes previous results of Carlson-Griffiths \cite[Theorem 2]{carlson1980infinitesimal} and of the author \cite[Proposition 6.1]{villaflor2021periods}.

\begin{thm}
\label{thm4}
Let $\P_\Sigma$ be a projective simplicial toric variety with anti-canonical class $\beta_0\in\Cl(\P_\Sigma)$. Let $X=\{F=0\}\subseteq\P_\Sigma$ be an ample quasi-smooth hypersurface of degree $\beta=\deg(F)\in \Cl(\P_\Sigma)$ and even dimension $n$. Let $P,Q\in S(\Sigma)^{\beta(\frac{n}{2}+1)-\beta_0}$, then
$$
\text{Tr}(\omega_P\cup\omega_Q)=-c\cdot \frac{(n+1)!}{(\frac{n}{2}!)^2}\text{Vol}(\Delta)
$$
where $\text{Vol}(\Delta)$ is the normalized volume of the convex polyhedron $\Delta\subseteq M_\R$ associated to $\beta$, and $c\in\C$ is the unique number such that 
$$
P\cdot Q\cdot x_1\cdots x_r\equiv c\cdot \text{Hess}_\Sigma(F) \hspace{5mm}(\text{mod }J_0(F)).
$$
Where $J_0(F)=\langle x_1\frac{\partial F}{\partial x_1},\ldots,x_r\frac{\partial F}{\partial x_r}\rangle\subseteq S({\Sigma})$ and $\Hess_{{\Sigma}}(F)$ is the toric Hessian of $F$ (see \cref{deftorichess}).
\end{thm}

We remark that in order to prove the aforementioned theorems it is also necessary to describe explicit differential forms associated to collections of divisor classes. In the case where the Picard rank of $\P_\Sigma$ is 1 these differential forms vanish, but in general they are non-zero and satisfy several identities with respect to contractions of Euler vector fields (see \cref{defEulervf}). Recall that since $\P_\Sigma$ is a GIT quotient of $\C^r:=\text{Spec }S(\Sigma)$ by the action of $\mathbb{D}(\Sigma):=\text{Spec }\C[\Cl(\P_\Sigma)]$ it is natural to lift differential forms of $\P_\Sigma$ to $\C^r$. Then a differential form in $\C^r$ descends to $\P_\Sigma$ if and only if it is $\mathbb{D}(\Sigma)$-invariant and vanishes when contracted by all Euler vector fields (see for instance \cite[Theorem 12.1]{batyrev1994hodge} and more generally \cite{brion1998differential} for GIT quotients by reductive group actions).

We also highlight that \cref{thm4} was obtained independently by Viergever \cite[Theorem 5.9]{Anneloes2023} in the case $\P_\Sigma$ is a product of projective spaces. Her approach is more general than ours in the sense that it applies for positive characteristic perfect fields.


The article is organized as follows. In \S \ref{sec2} we prove \cref{thm3} (see \cref{teomactor}). We apply this result in \S \ref{sec3} to study duality in Jacobian rings. In \S \ref{sec4} we introduce the Hodge structure in the cohomology of quasi-smooth varieties, we extend naturally the Hodge conjecture to this context and discuss the case of quasi-smooth intersections. We recall the description of the primitive Hodge filtration of quasi-smooth hypersurfaces in terms of residue forms by Batyrev and Cox \cite{batyrev1994hodge} in \S \ref{sec5} and explain the so called Cayley trick following Mavlyutov \cite{mavlyutov1999cohomology}. In \S \ref{sec6}
we recall classical results of Cattani, Cox and Dickenstein \cite{cox1996toric,cattani1997residues} about toric residues, we relate them to the computation of the trace map. In \S \ref{sec7} we recall Euler vector fields and their relation to toric differential forms. In \S \ref{sec8} we generalize classical results due to Carlson and Griffiths \cite{carlson1980infinitesimal} about computations of residues of meromorphic forms along quasi-smooth hypersurfaces to toric varieties, these results are used later to compute periods of algebraic cycles. In \S \ref{sec9} we introduce the differential forms associated to collections of divisor classes, necessary for the proof of \cref{mainthm}, \cref{thm2} and \cref{thm4}. In \S \ref{sec10} we use the coboundary map of the Poincare residue sequence to reduce \cref{thm2} to a trace computation in $\P_{\widetilde\Sigma}$. Putting all these ingredients together, in \S \ref{sec11}
we prove the main theorems. 

\bigskip

\noindent\textbf{Conventions and notations.} In order to avoid confusions, we will reserve some letters to always denote the same objects. In what follows we fix the notations we will be using along the article:
\begin{itemize}
    \item $\P_\Sigma$ will denote a projective simplicial toric variety given by the fan $\Sigma$. We reserve the letter $m:=\dim\P_\Sigma$ for its dimension. We denote by $\mathbb{T}(\Sigma)$ its dense open torus.
    \item $S(\Sigma)$ will denote the Cox ring of $\P_\Sigma$. We reserve the letter $r:=\#\Sigma(1)$ for the amount of rays, and denote $S(\Sigma)=\C[x_1,\ldots,x_r]$. Each variable $x_i$ is in correspondence with a $\mathbb{T}(\Sigma)$-invariant divisor $D_i$, which in turn is associated to a ray in $\Sigma(1)$ generated by a primitive element $\rho_i$.
    \item Every ample divisor is Cartier. 
    \item $\mathbb{D}(\Sigma):=\text{Spec }\C[\Cl(\P_\Sigma)]$ acts on $\C^r$ and $\P_\Sigma=\C^r\sslash \mathbb{D}(\Sigma)$ corresponds to the GIT quotient. 
    \item Given a $\mathbb{D}(\Sigma)$-invariant polynomial $F\in S(\Sigma)$, $X:=\{F=0\}\subseteq\P_\Sigma$ is the divisor obtained as $V(F)\sslash\mathbb{D}(\Sigma)$ for $V(F)\subseteq\C^r$.
    \item Given an homogeneous ideal $I\subseteq S(\Sigma)$, $R:=S(\Sigma)/I$ and $\alpha\in\Cl(\P_\Sigma)$, $R^\alpha$ will denote the degree $\alpha$ part of $R$. We reserve the letter $N$ to denote its socle degree (in case it has one), i.e. such that $\dim_\C R^N=1$.
    \item $\beta_0\in\Cl(\P_\Sigma)$ will denote the anti-canonical class of $\P_\Sigma$.
    \item $X\subseteq\P_\Sigma$ will denote a quasi-smooth subvariety, usually a quasi-smooth intersection of even dimension, and often non-degenerated. We reserve the letter $n:=\dim X$ for its dimension, and the letter $s:=m-n$ for its codimension. 
    \item In case $s=1$, i.e. $X\subseteq\P_\Sigma$ is a quasi-smooth hypersurface, we reserve the letter $F\in S(\Sigma)$ for its supporting equation $X=\{F=0\}$, and $\beta:=\deg(F)\in\Cl(\P_\Sigma)$ for its degree. In case $\Cl(\P_\Sigma)=\Z$ we will denote $\deg(F)=d$ instead.
    \item Given $X\subseteq\P_\Sigma$ a quasi-smooth hypersurface, we will denote by $J(F):=\langle \frac{\partial F}{\partial x_1},\ldots,\frac{\partial F}{\partial x_r}\rangle$ its Jacobian ideal and by $R(F):=S(\Sigma)/J(F)$ its Jacobian ring. We will also denote by $J_0(F):=\langle x_1\frac{\partial F}{\partial x_1},\ldots,x_r\frac{\partial F}{\partial x_r}\rangle$, $J_1(F):=(J_0(F):x_1\cdots x_r)$, $R_0(F):=S(\Sigma)/J_0(F)$ and $R_1(F):=S(\Sigma)/J_1(F)$.
    \item Given a quasi-smooth intersection $X\subseteq\P_\Sigma$, we denote by $Y\subseteq\P_{\widetilde{\Sigma}}$ the associated quasi-smooth hypersurface of the corresponding toric variety $\P_{\widetilde{\Sigma}}$ obtained by the Cayley trick. In such a case we denote $S(\widetilde{\Sigma})=\C[x_1,\ldots,x_r,x_{r+1},\ldots,x_{r+s}]$ and reserve the letters $f_1,\ldots,f_s\in S(\Sigma)$ for the defining equations of $X$, $\beta_i:=\deg(f_i)\in\Cl(\P_\Sigma)$ for their degrees, $F:=x_{r+1}f_1+\cdots+x_{r+s}f_s\in S(\widetilde{\Sigma})$ for the defining equation of $Y$, and $\beta:=\deg(F)\in\Cl(\P_{\widetilde{\Sigma}})$ for its degree.
    \item $\Omega$ denotes the standard top form of $\P_\Sigma$ explicitly described in \eqref{eqOmega}.
    \item Given a subset $I=\{i_1,\ldots,i_k\}\subseteq\{1,\ldots, r\}$ we denote $dx_I:=\bigwedge_{i\in I}dx_i$, $\widehat{dx_I}:=\bigwedge_{i\notin I}dx_i$, $\rho_I=(\rho_{i_1}|\cdots|\rho_{i_k})$ and given some $F\in S(\Sigma)^\beta$ we denote $F_I:=\prod_{i\in I}\frac{\partial F}{\partial x_i}$.
    \item Given a vector field $v$ in $\C^r$ we denote $\iota_v$ its contraction operator. Given a differential form $\eta$ in $\C^r$ and a subset $J=\{j_1,\ldots,j_l\}\subseteq\{1,\ldots, r\}$ we denote by $\eta_J:=\iota_{\frac{\partial}{\partial x_{j_l}}}(\cdots\iota_{\frac{\partial}{\partial x_{j_0}}}(\eta)\cdots)$.
    \item We denote by $E_j$ the Euler vector fields in $\C^r$ (see \cref{defEulervf}). For a collection $\{j_1,\ldots,j_p\}\subseteq \{m+1,\ldots,r\}$ we denote $\widehat{\iota_{E_{j_1,\ldots,j_p}}}:=\iota_{E_r}\cdots\widehat{\iota_{E_{j_p}}}\cdots\widehat{\iota_{E_{j_1}}}\cdots\iota_{E_{m+1}}$. 
    \item We denote by $V:=dx_1\wedge\cdots\wedge dx_r$, and given a collection of divisor classes $\alpha_1,\ldots,\alpha_k\in\Cl(\P_\Sigma)$ we denote by $V^{\alpha_1,\ldots,\alpha_k}$ the differential form in $\C^r$ given by \cref{defValphais}.
\end{itemize}

\noindent\textbf{Acknowledgements.} Part of this work was inspired by conversations and correspondence with Andreas P. Braun, Alicia Dickenstein, Hossein Movasati and Giancarlo Urzúa. I also got benefited from discussions with Jorge Vitório Pereira, Jorge Duque Franco, Hugo Fortin and William D. Montoya. Special thanks go to Pedro Montero and Sebastián Velazquez for several useful discussions and references. Financial support was provided by Fondecyt ANID postdoctoral grant 3210020.

\section{Toric Macaulay theorem}
\label{sec2}
In this section we prove \cref{thm3} which is a generalization of a classical theorem due to Macaulay \cite{mac16} to the toric setup. We focus on the case where the ambient is a projective simplicial toric variety. 

\begin{dfn}
\label{defregular} 
Let $\P_\Sigma$ be an $m$-dimensional projective simplicial toric variety with Cox ring $S(\Sigma)=\C[x_1,\ldots,x_r]$ where $r=\#\Sigma(1)$. For a homogeneous ideal $I\subseteq S(\Sigma)$ we say $N\in\Cl(\P_\Sigma)$ is a \textit{socle degree} of $I$ if $\dim_\C R^N=1$ for $R:=S(\Sigma)/I$. For such a pair $(I,N)$ we say that a class $\beta\in \Cl(\P_\Sigma)$ is \textit{regular} for $(I,N)$ if the multiplication map 
$$ 
R^\beta\times R^{N-\beta}\rightarrow R^N\simeq \C 
$$
induces an injection $R^\beta\hookrightarrow (R^{N-\beta})^\vee$.
\end{dfn}

\begin{rmk}
Every non effective class $\beta\in \Cl(\P_\Sigma)$ is regular, since $S(\Sigma)^\beta=0$ and so $R^\beta=0$.
\end{rmk}


\begin{dfn}
Let $\alpha_0,\ldots,\alpha_m\in\Cl(\P_\Sigma)$ be ample classes. We say that a class $\beta\in\Cl(\P_\Sigma)$ \textit{satisfies the vanishing condition ($*$) for $(\alpha_0,\ldots,\alpha_m)$} if for any $q\in \{1,2,\ldots,m-1\}$ and any $0\le i_1<i_2<\cdots<i_{q+1}\le m$
\begin{equation}
\label{condstar}
H^q\left(\O_{\P_\Sigma}\left(\beta-\sum_{j=1}^{q+1}\alpha_{i_j}\right)\right)=0.\tag{$*$}    
\end{equation}
\end{dfn}

\begin{thm}
\label{teomactor}
Let $\alpha_0,\ldots,\alpha_m\in\Cl(\P_\Sigma)$ be ample classes and let $f_0,\ldots,f_m\in S(\Sigma)$ be such that $\deg(f_i)=\alpha_i$ and $V(f_0,\ldots,f_m)=\varnothing\subseteq\P_\Sigma$. Let $I:=\langle f_0,\ldots,f_m\rangle$ and $N:=(\sum_{i=0}^m\alpha_i)-\beta_0$, where $\beta_0$ is the anti-canonical class. Then $N$ is a socle degree of $I$, and every Cartier class $\beta\in\Cl(\P_\Sigma)$ satisfying the vanishing condition ($*$) for $(\alpha_0,\ldots,\alpha_m)$ is regular for $(I,N)$.
\end{thm}

\begin{proof}
Note first that $\beta=0$ or $N$ also satisfy the vanishing condition ($*$) by Batyrev-Borisov vanishing \cite[Theorem 9.2.7]{cox2011toric}. Consider the vector bundle
$$
\mathcal{E}:=\bigoplus_{i=0}^m\O_{\P_\Sigma}(-\alpha_i).
$$
The natural map induced by $(f_0,\ldots,f_m):\mathcal{E}\rightarrow\O_{\P_\Sigma}$ induces a Koszul complex $\mathcal{E}^\bullet$ where
$$
\mathcal{E}^p:=\bigwedge^{m+1-p}\mathcal{E} \ , \hspace{1cm}\text{ for }0\le p\le m+1,
$$
and $\mathcal{E}^p:=0$ otherwise. Since $\P_\Sigma$ is Cohen-Macaulay, $\mathcal{E}^\bullet$ is exact, then $\mathcal{E}^\bullet(\beta):=\mathcal{E}^\bullet\otimes\O_{\P_\Sigma}(\beta)$ is also exact and so 
$$
\Hip^k(\mathcal{E}^\bullet(\beta))=0 \ , \hspace{1cm}\forall k\ge 0.
$$
Let $E^{p,q}_r$ be the spectral sequence associated to the naive filtration of $\mathcal{E}^\bullet(\beta)$. Since $\beta$ satisfies the vanishing condition ($*$) for $(\alpha_0,\ldots,\alpha_m)$ and $\mathcal{E}^p(\beta)=\bigoplus_{|I|=m+1-p}\O_{\P_\Sigma}\left(\beta-\sum_{i\in I}\alpha_i\right)$ it follows that
$$
E_2^{m-q,q}=H^{m-q}(E^{\bullet,q}_1,d_1)=H^{m-q}(H^q(\mathcal{E}^\bullet(\beta)),d_1)=0
$$
for all $q=1,\ldots,m-1$. Hence $E^{m+1,0}_2=E^{m+1,0}_3=\cdots=E^{m+1,0}_{m+1}$. Since the spectral sequence degenerates at $E_{m+2}$ we conclude that
$$
d_{m+1}:E_{m+1}^{0,m}\rightarrow E_{m+1}^{m+1,0}
$$
is an isomorphism. Noting that 
$$
E_2^{m+1,0}=\text{Coker}(H^0(\mathcal{E}(\beta))\rightarrow H^0(\O_{\P_\Sigma}(\beta)))=R^\beta
$$
and (by Serre duality)
$$
E_2^{0,m}=\text{Ker}(H^m(\O_{\P_\Sigma}(\beta-\sum_{i=0}^m\alpha_i))\rightarrow H^m(\mathcal{E}^\vee(\beta-\sum_{i=0}^m\alpha_i)))
$$
$$
=(\text{Coker}(H^0(\mathcal{E}(N-\beta))\rightarrow H^0(\O_{\P_\Sigma}(N-\beta))))^\vee=(R^{N-\beta})^\vee
$$
one sees that $d_{m+1}: E^{0,m}_{m+1}\xrightarrow{\sim} R^\beta$. Since $E^{0,m}_{m+1}\subseteq E^{0,m}_2$, we conclude for $\beta=0$ and $N$ that $N$ is a socle degree of $I$. Furthermore for any other $\beta$ satisfying the vanishing condition ($*$), the inclusion $E^{0,m}_{m+1}\subseteq E^{0,m}_2$ corresponds to $d_{m+1}$ composed with the map induced by multiplication $R^\beta\rightarrow (R^{N-\beta})^\vee$, hence $\beta$ is regular for $(I,N)$.
\end{proof}

\begin{rmk}
In particular, we recover the classical Macaulay theorem in the case $\P_\Sigma$ is a weighted projective space. In fact, it follows from \cref{teomactor} and Bott's vanishing \cite{dolgachev1982weighted} that all classes are regular, hence $R=\C[x_0,\ldots,x_m]/\langle f_0,\ldots, f_m\rangle$ is an Artinian Gorenstein algebra. Note that if $\P_\Sigma$ has Picard rank bigger than 1, there are infinitely many non-regular classes. In fact, since $R$ is not Artinian and there are only finitely many effective classes $\beta$ such that $N-\beta$ is also effective \cite[Lemma 4.6]{lazarsfeld2009convex}, it follows that $R^\beta\neq 0$ and $R^{N-\beta}= 0$ for infinitely many classes. 
\end{rmk}

\section{Regular divisor classes on Jacobian rings}
\label{sec3}
In this section we transport some properties of Artinian Gorenstein ideals of a polynomial ring to other Cox rings. As pointed out in \S \ref{sec2}, there are important differences when $\P_\Sigma$ has Picard rank bigger than 1, since the natural analogues are not anymore Artinian, and they do not satisfy the perfect pairing property either, for most degrees. Nevertheless, for the case of Jacobian ideals several degrees will be regular.



\begin{prop}
\label{propcoxgor}
Let $\P_\Sigma$ be a projective simplicial toric variety with Cox ring $S(\Sigma)$. Let $I\subseteq S(\Sigma)$ be a homogeneous ideal, $N\in\Cl(\P_\Sigma)$ be a socle degree of $I$ and let $R:=S(\Sigma)/I$.
\begin{itemize}
    \item[(a)] If $J\subseteq I$ is an homogeneous ideal of socle degree $N$, then $J^\beta=I^\beta$ for all $\beta\in\Cl(\P_\Sigma)$ regular for $(J,N)$.
    \item[(b)] If $\alpha\in\Cl(\P_\Sigma)$ is regular for $(I,N)$ and $f\in S(\Sigma)^\alpha\setminus I^\alpha$, then $N-\alpha$ is a socle degree for 
    $$
    (I:f):=\{g\in S(\Sigma) : \ fg\in I\}.
    $$
    Furthermore, for any $\beta\in\Cl(\P_\Sigma)$ regular for $(I,N)$, $\beta-\alpha$ is regular for $((I:f),N-\alpha)$.
\end{itemize}
\end{prop}

\begin{proof} 
\begin{itemize}
    \item[(a)] For every $f\in I^\beta$, $fg\in I^N=J^N$ for all $g\in S(\Sigma)^{N-\beta}$. In other words, if $R:=S(\Sigma)/J$ then $f\in R^\beta\mapsto 0\in (R^{N-\beta})^\vee$. Since $\beta$ is regular for $(J,N)$ it follows that $f=0\in R^\beta$, i.e. $f\in J^\beta$.
    \item[(b)] Let $R_1:=S(\Sigma)/I$ and $R_2:=S(\Sigma)/(I:f)$. To see that $N-\alpha$ is a socle degree for $(I:f)$ consider the map
    $$
    R_2^{N-\alpha}\xrightarrow{\cdot f}R_1^N.
    $$
    This map is clearly injective. On the other hand, since $\alpha$ is regular for $(I,N)$ and $f\notin I^\alpha$, there exists some $g\in S(\Sigma)^{N-\alpha}$ such that $fg\notin I^N$. This shows that the map above is not zero, hence it is an isomorphism of $1$-dimensional spaces. 
    
    For the last assertion, consider any $g\in S(\Sigma)^{\beta-\alpha}$ such that $gh\in (I:f)^{N-\alpha}$ for all $h\in S(\Sigma)^{N-\beta}$. Then $fgh\in I^N$ for all $h\in S(\Sigma)^{N-\beta}$. Since $\beta$ is regular for $(I,N)$, it follows that $fg\in I^\beta$, i.e. $g\in (I:f)^{\beta-\alpha}$.
\end{itemize}
\end{proof}

Let us turn now to the analysis of regular degrees for Jacobian ideals in the non-degenerated case.

\begin{dfn}
\label{defnondegint}
We say a subvariety $X\subseteq \P_\Sigma$ is \textit{non-degenerated} if $X\cap \mathbb{T}_\tau$ is a smooth subvariety of $\mathbb{T}_\tau$ for every $\tau\in\Sigma$, where $\mathbb{T}_\tau$ is the torus orbit corresponding to $\tau$.
\end{dfn}

The following result due to Cox allows us to study regularity on some alternative Jacobian ideals (which we denote $J_0(F)$ and $J_1(F)$), in spite that the classical Jacobian ideal $J(F)=\langle \frac{\partial F}{\partial x_1},\ldots,\frac{\partial F}{\partial x_r}\rangle$ is not a complete intersection if the ambient toric variety $\P_\Sigma$ is not a weighted projective space.

\begin{prop}[Cox]
\label{propj0cox}
Let $\P_\Sigma$ be an $m$-dimensional projective simplicial toric variety. Let $\beta\in\Cl(\P_\Sigma)$ be a $\Q$-ample class, and let $F\in S(\Sigma)=\C[x_1,\ldots,x_r]$ with $\deg(F)=\beta$ be such that $X:=\{F=0\}\subseteq\P_\Sigma$ is non-degenerated. Let $J_0(F):=\langle x_1\frac{\partial F}{\partial x_1},\ldots,x_r\frac{\partial F}{\partial x_r}\rangle$, and suppose $\rho_{i_1},\ldots,\rho_{i_m}$ are linearly independent primitive generators of the rays of $\Sigma$. Then $V(F,x_{{i_1}}\frac{\partial F}{\partial x_{i_1}},\ldots,x_{i_m}\frac{\partial F}{\partial x_{i_m}})=\varnothing\subseteq\P_\Sigma$ and
$$
J_0(F)=\left\langle F,x_{i_1}\frac{\partial F}{\partial x_{i_1}},\ldots,x_{i_m}\frac{\partial F}{\partial x_{i_m}}\right\rangle.
$$
\end{prop}

\begin{proof}
See \cite[Proposition 5.3]{cox1996toric}.
\end{proof}

\begin{prop}
\label{propunimodj0}
Let $X=\{F=0\}\subseteq \P_\Sigma$ be an ample non-degenerated hypersurface of an $m$-dimensional projective simplicial toric variety $\P_\Sigma$. Let $S(\Sigma)=\C[x_1,\ldots,x_r]$ be its Cox ring, with $r=\#\Sigma(1)$. Let $\beta:=\deg(F)\in\Cl(\P_\Sigma)$ and $\beta_0\in\Cl(\P_\Sigma)$ be the anti-canonical class. Then $s\beta+t\beta_0$ satisfies the vanishing condition ($*$) for $(\beta,\ldots,\beta)$, hence is regular for $(J_0(F),(m+1)\beta-\beta_0)$, for any $s\in\Z$ and any $t\in\{-1,0,1,2\}$.
\end{prop}

\begin{proof}
For the case $t=0$, it is clear that $s\beta$ satisfies the vanishing condition ($*$) for $(\beta,\ldots,\beta)$ by Mavlyutov's vanishing theorem \cite[Theorem 9.3.3]{cox2011toric} and Batyrev-Borisov's vanishing \cite[Theorem 9.2.7]{cox2011toric}. Using Serre duality we get the result for $t=1$. The case $t=-1$ follows from Bott-Steenbrink-Danilov's vanishing \cite[Theorem 9.3.1]{cox2011toric} and Batyrev-Borisov's vanishing. And applying Serre duality to this case we get the result for $t=2$. The regularity of each degree follows by \cref{teomactor}.
\end{proof}

\begin{rmk}
Several other degrees can be shown to be regular for $(J_0(F),(m+1)\beta-\beta_0)$ by applying Mustaţă's vanishing \cite{mustata2002vanishing} and Serre duality.
\end{rmk}

\begin{cor}
\label{corunimodj1}
In the same context of \cref{propunimodj0}, if $x_1\cdots x_r\notin J_0(F)$ then $s\beta+t\beta_0$ is regular for $(J_1(F),(m+1)\beta-2\beta_0)$ for any $s\in\Z$ and $t\in\{-2,-1,0,1\}$.
\end{cor}

\begin{proof}
Apply \cref{propcoxgor} (b) to $J_1(F)=(J_0(F):x_1\cdots x_r)$. \end{proof}

\begin{rmk}
Later in \S \ref{sec5} we will describe the primitive cohomology of $X$ and show that if $H^{m-1}(X,\C)_\prim\neq 0$ then some $R_1(F)^{(m-p)\beta-\beta_0}\simeq H^{p,m-p-1}(X)_\prim\neq 0$. In consequence, whenever there exists some non-trivial primitive class, the condition $x_1\cdots x_r\notin J_0(F)$ is satisfied.
\end{rmk}

\section{Hodge structure}
\label{sec4}
Let $X\subseteq \P_\Sigma$ be an $n$-dimensional quasi-smooth subvariety of a projective simplicial toric variety $\P_\Sigma$. It follows from the toric Chow lemma \cite[Theorem 6.1.18]{cox2011toric} that $X$ is a complete almost K\"ahler orbifold (or $V$-manifold) in the sense of \cite[\S 2.5]{peters2008mixed}. In fact there exist a smooth projective variety $M$ and a dominant birational map $\phi:M\rightarrow X$. Hence $H^k(X,\Q)$ carries a natural pure Hodge structure of weight $k$, induced by the inclusion $\phi^*:H^k(X,\Q)\hookrightarrow H^k(M,\Q)$ as a Hodge substructure of $H^k(M,\Q)$ \cite[Theorem 2.43]{peters2008mixed}. Moreover
$$
Gr_F^pH^k(X,\C)\simeq H^q(X,\widetilde{\Omega}_X^p)=H^q(X^{sm},\Omega_{X^{sm}}^p),
$$
where $i:X^{sm}\hookrightarrow X$ is the smooth locus and $\widetilde{\Omega}_X^p=i_*\Omega_{X^{sm}}^p$ is the sheaf of forms locally invariant under the finite group action around each singular point of $X$. In other words, for every point $x\in X$ there exists a small subgroup $G<\GL_n(\C)$ and an open ball $B\subseteq \C^n$ such that $\mathcal{O}^{hol}_{X,x}=(\mathcal{O}_{B,0})^G$. Thus we have that $\widetilde{\Omega}_{X,x}^p\otimes \mathcal{O}^{hol}_{X,x}=(\Omega_{B,0}^p)^G$ is the germ of $G$-invariant analytic $p$-forms. 

Similarly, we can define the de Rham cohomology groups on $X$ as
$$
H^k_\dR(X):=H^k(\Gamma(\widetilde{\Omega}^\bullet_{X^\infty}),d)\simeq H^k(X,\C).
$$
Where $(\widetilde{\Omega}^\bullet_{X^\infty},d)$ is the $\Gamma$-acyclic resolution of the constant sheaf $\C$ over the analytic topology of $X$, given by the sheaves of $C^\infty$ complex valued differential forms compatible with the orbifold structure of $X$, i.e. which are locally such that $\widetilde{\Omega}_{X^\infty,x}^k=(\Omega_{B^\infty,0}^k)^G$. The map $\phi:M\rightarrow X$ induces also an inclusion $\phi^*:H^k_\dR(X)\hookrightarrow H^k_\dR(M)$, providing a natural Hodge decomposition on $H^k_\dR(X)$ by
$$
H^{p,q}(X):=H^{p,q}(M)\cap H^k_\dR(X)\simeq H^q(X,\widetilde{\Omega}_{X}^p).
$$
\begin{dfn}
We say that a rational class $\eta\in H^{2p}(X,\Q)$ is a \textit{Hodge cycle} if $\eta\in F^pH^{2p}(X,\C)\simeq \Hip^{2p}(X,\widetilde{\Omega}_X^{\bullet\ge p})$.
\end{dfn}

Since such a variety $X$ is complete and rationally smooth \cite[Definition 11.4.3, Example 11.4.4]{cox2011toric}, it satisfies both Poincar\'e dualities with $\Q$-coefficients \cite[\S 12.4]{cox2011toric}
$$
H_k(X,\Q)\simeq H^k(X,\Q)^\vee\simeq H^{2n-k}(X,\Q).
$$
Every codimension $p$ algebraic subvariety $W\subseteq X$ induces (by triangulation) an homology cycle in $\delta_W=\phi_*\delta_V\in H_{2n-2p}(X,\Q)$, where $V\subseteq Y$ is the strict transform of $W$ under $\phi$. This in turn induces a cohomology class
$$
[W]\in H^{2p}(X,\Q).
$$
Such a class $[W]$ is called an \textit{algebraic cycle}. Since $[W]\in H^{2p}(X,\Q)\mapsto [V]\in H^{2p}(Y,\Q)\cap H^{p,p}(Y)$, it follows that every algebraic cycle of $X$ is a Hodge cycle. Thus we can extend to this setup the Hodge conjecture and ask whether all Hodge cycles in $X$ are algebraic cycles or not. Similarly we can extend the variational Hodge conjecture and ask whether in all proper families $\pi:X\rightarrow T$ of quasi-smooth subvarieties of $\P_\Sigma$, with connected base $T$, every flat section of $R^{2p}\pi_*\Omega_{X/T}^{\bullet\ge p}$ is algebraic at one point if and only if it is algebraic everywhere. 

\begin{rmk}
The Hodge conjecture is known to hold for $\P_\Sigma$ a projective simplicial toric variety. In fact the cohomology ring is generated by $\mathbb{T}(\Sigma)$-invariant algebraic cycles \cite[Theorem 12.5.3]{cox2011toric}. 
\end{rmk}

As in the smooth projective case, the Hodge conjecture for complete intersections is only non-trivial for the middle cohomology group of even dimensional varieties.

\begin{prop}
Let $X\subseteq \P_\Sigma$ be an ample quasi-smooth hypersurface of a projective simplicial toric variety $\P_\Sigma$, then the Gysin map
$$
H^{k}(X,\Q)\xrightarrow{i!} H^{k+2}(\P_\Sigma,\Q)
$$
is an isomorphism for $k>\dim X$.
\end{prop}

\begin{proof}
Consider the Leray spectral sequence for the map $i:X\hookrightarrow \P_\Sigma$. Since $X$ is rationally smooth \cite[Definition 11.4.3, Example 11.4.4]{cox2011toric}, the Leray sequence degenerates at $E_3$ inducing the Leray exact sequence $$
\cdots\rightarrow H^{k+1}(U,\Q)\rightarrow H^{k}(X,\Q)\xrightarrow{i!} H^{k+2}(\P_\Sigma,\Q)\rightarrow H^{k+2}(U,\Q)\rightarrow\cdots
$$
for $U:=\P_\Sigma\setminus X$. As $U$ is affine it follows that $\widetilde{\Omega}_U^\bullet$ is an acyclic complex, hence $H^\ell(U,\C)=\Hip^\ell(U,\widetilde{\Omega}_U^\bullet)=H^\ell(\Gamma(\widetilde{\Omega}_U^\bullet),d)=0$ for $\ell>\dim U$. Thus $H^\ell(U,\Q)=0$ for $\ell>\dim U$.
\end{proof}

\begin{cor}
Let $X\subseteq \P_\Sigma$ be a quasi-smooth intersection inside a simplicial projective toric variety $\P_\Sigma$, then $H^{2k}(X,\Q)$ is generated by algebraic cycles for $2k\neq \dim X$.
\end{cor}

\begin{proof}
Applying inductively the argument of the previous proposition one gets that the composition of Gysin maps
$
H^\ell(X,\Q)\rightarrow H^{\ell+2m}(\P_\Sigma,\Q)
$
is an isomorphism for $\ell>\dim X$. Since the Gysin map is dual to the intersection map in homology, it is also an isomorphism of algebraic cycles.
For $2k<\dim X$ apply the Hard Lefschetz theorem \cite[Theorem 12.5.8]{cox2011toric}.
\end{proof}

\section{Cohomology of quasi-smooth intersections}
\label{sec5}
In this section we will recall the description of the primitive part of the middle de Rham cohomology group of an even dimensional quasi-smooth intersection in terms of residues. These are classical results taken from \cite{batyrev1994hodge} and \cite{mavlyutov1999cohomology}.

\begin{dfn}
Let $\P_\Sigma$ be a projective simplicial toric variety of dimension $m$ and $X\subseteq \P_\Sigma$ be a quasi-smooth hypersurface given by $F\in S^\beta$. We define the \textit{$(m-1)$-th primitive cohomology} group of $X$ as
$$
H^{m-1}(X,\Q)_\prim:=\ker(H^{m-1}(X,\Q)\xrightarrow{i_!}H^{m+1}(\P_\Sigma,\Q)).
$$
\end{dfn}

Considering the commutative diagram
\begin{center}
\begin{tikzpicture}[xscale=1.5,yscale=1]

\path       
      node   (m22) at (0,2) {$H^{m-1}(X,\Q)$} 
      node   (m23) at (2,2) {$H^{m+1}(\P_\Sigma,\Q)$} 
      node   (m33) at (1,0.5) {$H^{m-1}(\P_\Sigma,\Q)$};
      { 
      \draw[->]   (m33)  edge node[left] {$i^*$} (m22);
      \draw[->]   (m33)  edge node[right] {$L$} (m23);
      \draw[->]   (m22)  edge node[above] {$i_!$} (m23);
      
       }
\end{tikzpicture}
\end{center}
we see that by Hard Lefschetz theorem $L$ is an isomorphism and so
$$
H^{m-1}(X,\Q)\simeq H^{m-1}(X,\Q)_\prim\oplus H^{m-1}(\P_\Sigma,\Q).
$$

\begin{thm}[Batyrev-Cox]
\label{thmbatycox}
Let $\P_\Sigma$ be a projective simplicial toric variety of dimension $m$. Let $X=\{F=0\}\subseteq \P_\Sigma$ be a quasi-smooth hypersurface corresponding to an ample divisor with $\deg(F)=\beta\in \Cl (\P_\Sigma)$. Then for $p\neq \frac{m}{2}-1$ we have an isomorphism
$$
R(F)^{(m-p)\beta-\beta_0}\simeq H^{p,m-p-1}(X)_\prim
$$
$$
P\mapsto \res\left(\frac{P\Omega}{F^{m-p}}\right)^{p,m-p-1}
$$
where $R(F)=S(\Sigma)/J(F)$ is the Jacobian ring associated to $F$ with the grading induced by the Cox ring $S(\Sigma)$, $\beta_0\in \Cl(\P_\Sigma)$ is the anti-canonical class, and $\Omega\in H^0(\P_\Sigma,\widetilde{\Omega}_{\P_\Sigma}^m(\beta_0))$ is the canonical generator (see \eqref{eqOmega} for an explicit description).
\end{thm}

\begin{proof}
See \cite[Theorem 10.13]{batyrev1994hodge}.
\end{proof}

\begin{rmk}
\label{rmkCayleytrick}
In the case of an $n$-dimensional quasi-smooth intersection of ample divisors $X=\{f_1=f_2=\cdots=f_s=0\}\subseteq\P_\Sigma$ with $s=m-n$, we use the Cayley trick \cite[\S 2]{mavlyutov1999cohomology} to construct a toric map between projective simplicial toric varieties
$$
\pi: \P_{\widetilde\Sigma}=\P(\mathcal{E})\rightarrow\P_\Sigma
$$
where $\mathcal{E}=L_1\oplus\cdots\oplus L_s$ and $\deg(f_i)=[L_i]\in \Cl(\P_\Sigma)$. The coordinate ring of $\P(\mathcal{E})$ is of the form $S(\widetilde\Sigma)=\C[x_1,\ldots,x_r,x_{r+1},\ldots,x_{r+s}]$ where $\C[x_1,\ldots,x_r]=S(\Sigma)$ is the coordinate ring of $\P_\Sigma$. Considering $Y=\{F=0\}\subseteq\P(\mathcal{E})$ where
$$
F:=x_{r+1}\cdot f_1+x_{r+2}\cdot f_2+\cdots+x_{r+s}\cdot f_s,
$$
one can show that $Y$ is quasi-smooth and
$$
\pi:\P(\mathcal{E})\setminus Y\rightarrow \P_\Sigma\setminus X
$$
is a $\C^{s-1}$ bundle in the Zariski topology.
\end{rmk}

\begin{thm}[Mavlyutov]
\label{thmmavly}
Let $\P_\Sigma$ be an $m$-dimensional projective simplicial toric variety, and let $X\subseteq\P_\Sigma$ be a quasi-smooth intersection of ample divisors defined by $f_i\in S^{\beta_i}$, $i=1,\ldots,s$. Let $F:=x_{r+1}f_1+\cdots+x_{r+s}f_s\in S(\widetilde{\Sigma})$, then for $p\neq \frac{m+s-1}{2}$ we have a canonical isomorphism
$$
R(F)^{(m+s-p)\beta-\beta_0}\simeq H^{p-s,m-p}(X)_\prim,
$$
$$
P\mapsto \omega_P
$$
where $\beta_0$ is the anti-canonical class of $\P(\mathcal{E})$.
\end{thm}

\begin{proof}
See \cite[Theorem 3.6]{mavlyutov1999cohomology}. The isomorphism follows from
$$
H^{p-s,m-p}(X)_\prim\simeq \Gr^p_FH^{m+s-1}(\P_\Sigma\setminus X)\simeq \Gr^p_FH^{m+s-1}(\P(\mathcal{E})\setminus Y)\simeq H^{p-1,m+s-p-1}(Y)_{\prim} 
$$
and \cref{thmbatycox}.
\end{proof}

\begin{rmk}
\label{rmkCT2}
From the proof of the theorem above we get an isomorphism of Hodge structures of type $(s-1,s-1)$
$$
H^{n}(X,\Q)_\prim\simeq H^{m+s-2}(Y,\Q)_\prim.
$$
When $X$ is even dimensional, we have an isomorphism of the spaces of Hodge cycles
$$
H^{\frac{n}{2},\frac{n}{2}}(X)_\prim\cap H^{n}(X,\Q)\simeq H^{\frac{m+s-2}{2},\frac{m+s-2}{2}}(Y)_\prim\cap H^{m+s-2}(Y,\Q).
$$
Moreover the above map preserves algebraic cycles. This follows from the fact that we can factorize it as
\begin{equation}
\label{eqCaytr}
H^{n}(X,\Q)_\prim\xrightarrow{\pi^*} H^{n}(\P(\mathcal{E})|_X,\Q)\xrightarrow{j!}H^{m+s-2}(Y,\Q)
\end{equation}
where $j!$ is the Gysin map associated to the inclusion $j:\P(\mathcal{E})|_X\hookrightarrow Y$. Hence, by duality we see that the primitive part of the class of an algebraic cycle $Z\in \CH^\frac{n}{2}(X)$ is mapped to $[\pi^{-1}(Z)]_\prim\in H^{\frac{m+s-2}{2},\frac{m+s-2}{2}}(Y,\Q)_\prim$. In particular, it takes complete intersection cycles (with respect to $\P_\Sigma$) to complete intersection cycles (with respect to $\P(\mathcal{E})$).
\end{rmk}

\begin{rmk}
In the non-degenerated case we can replace $R(F)$ by $R_1(F)$ in \cref{thmbatycox} (see \cite[Theorem 11.8]{batyrev1994hodge}) and \cref{thmmavly} (see \cite[Lemma 4.3]{mavlyutov1999cohomology}).
\end{rmk}

\section{Trace map}
\label{sec6}
In this section we recall some results from \cite{cox1996toric, cattani1997residues} about the computation of Grothendieck residues in toric varieties. Computing these residues is equivalent to compute periods of top forms of $\P_\Sigma$, i.e. to compute the trace map.

Let $\P_\Sigma$ be a projective simplicial toric variety of dimension $m$. The trace map corresponds to
$$
\text{Tr}: \omega\in H_\dR^{2m}(\P_\Sigma)\xrightarrow{\sim}\frac{1}{(2\pi i)^m}\int_{\P_\Sigma}\omega\in\C. 
$$
Since $H^{2m}_\dR(\P_\Sigma)\simeq H^{m,m}(\P_\Sigma)\simeq H^m(\P_\Sigma,\widetilde{\Omega}_{\P_\Sigma}^m)$, whenever we have ample classes $\alpha_0,\ldots,\alpha_m\in\Cl(\P_\Sigma)$ and take $f_0,\ldots,f_m\in S(\Sigma)$ such that $\deg(f_i)=\alpha_i$ and $V(f_0,\ldots,f_m)=\varnothing\subseteq\P_\Sigma$, there exists a canonical isomorphism using Cech cohomology (by \cref{teomactor})
$$
P\in R^N\xrightarrow{\sim} \xi_P:=\frac{P\Omega}{f_0\cdots f_m}\in H^m(\P_\Sigma,\widetilde{\Omega}_{\P_\Sigma}^m). 
$$
Where $R:=S(\Sigma)/\langle f_0,\ldots,f_m\rangle$, $N:=(\sum_{i=0}^m\alpha_i)-\beta_0$ is a socle degree, and $\Omega$ is the generator of $H^0(\P_\Sigma,\widetilde{\Omega}_{\P_\Sigma}^m(\beta_0))$ given by
\begin{equation}
\label{eqOmega}
\Omega:=\sum_{|I|=m}\det(\rho_I)\widehat{x_I}dx_I,
\end{equation}
where $dx_I:=dx_{i_1}\wedge\cdots\wedge dx_{i_m}$, $\widehat{x_I}:=\prod_{i\notin I}x_i$ and $\det(\rho_I):=\det(\rho_{i_1}|\cdots|\rho_{i_m})$ with $\rho_i$ primitive generators of the rays of $\Sigma$. Thus, the problem of computing the trace can be reduced to determine a generator $Q$ of $R^N$ such that $\xi_{Q}$ has trace 1, and then the trace of any $\xi_P$ is the unique number $c\in\C$ such that $P-c\cdot Q\in\langle f_0,\ldots,f_m\rangle$.


This problem was treated by Cattani, Cox and Dickenstein, the most general result available is \cite[Theorem 0.2]{cattani1997residues}. In the particular case where $\deg(f_0)=\deg(f_1)=\cdots=\deg(f_m)$, Cox \cite[Theorem 5.1]{cox1996toric} found an explicit formula for computing the trace of a toric Jacobian. Note that we do a correction of Cox formula by adding the sign $(-1)^{m+1\choose 2}$. In the subsequent work of Cattani, Cox and Dickenstein \cite{cattani1997residues}, they leave a sign ambiguity. If one examines Cox's proof, the sign must appear from $\text{Tr}_{\P^m}([\omega_1])$, which is not 1 by \cite[Proposition 4.1]{villaflor2021periods}. This sign in the trace of the canonical top form of $\P^m$ was at first time pointed out by Deligne \cite[page 6]{dmos}.

\begin{dfn}
\label{deftoricjac}
Let $\P_\Sigma$ be an $m$-dimensional projective simplicial toric variety. Let $\beta\in\Cl(\P_\Sigma)$ be an ample class and let $f_0,\ldots,f_m\in S(\Sigma)^\beta$ be such that $V(f_0,\ldots,f_m)=\varnothing\subseteq\P_\Sigma$. We define the \textit{toric Jacobian} of $(f_0,\ldots,f_m)$ as $Q\in S(\Sigma)^{(m+1)\beta-\beta_0}$ such that
$$
\sum_{i=0}^m(-1)^if_i df_0\wedge\cdots\widehat{df_i}\cdots\wedge df_m=Q\Omega.
$$
We denote it as $\Jac_\Sigma(f_0,\ldots,f_m):=Q$.
\end{dfn}

\begin{thm}[Cox]
\label{thmcox}
In the context of the previous definition. Let $\rho_{i_1},\ldots,\rho_{i_m}$ be linearly independent primitive generators of the rays of $\Sigma$. Then the toric Jacobian is
$$
\Jac_\Sigma(f_0,\ldots,f_m)=\frac{\det\begin{pmatrix}f_0 & \cdots & f_m \\ \partial f_0/\partial x_{i_1} & \cdots & \partial f_m/\partial x_{i_1}\\
\vdots &  & \vdots \\
\partial f_0/\partial x_{i_m} & \cdots & \partial f_m/\partial x_{i_m} \end{pmatrix}}{\det(\rho_I)\widehat{x_I}}
$$
and 
$$
\text{Tr}([\xi_{\Jac_\Sigma(f_0,\ldots,f_m)}])=(-1)^{m+1\choose 2} m!\text{Vol}(\Delta)=(-1)^{m+1\choose 2}\deg(f),
$$
where $\text{Vol}(\Delta)$ is the normalized volume of the convex polyhedron $\Delta\subseteq M_\R$ associated to $\beta$ and $f:\P_\Sigma\rightarrow \P^m$ is the map defined by $f:=(f_0:\cdots:f_m)$.
\end{thm}


\begin{rmk}
In the particular case of $\P_\Sigma=\P^m$ and $X=\{F=0\}$ be a degree $d$ smooth hypersurface, one can consider the Jacobian ideal $J(F):=\langle \frac{\partial F}{\partial x_0},\ldots,\frac{\partial F}{\partial x_m}\rangle\subseteq\C[x_0,\ldots,x_m]$. By smoothness $V(J(F))=\varnothing$, thus the generator of $R(F)^N$, where $R(F):=\C[x_0,\ldots,x_m]/J(F)$ is the Jacobian ring and $N=(d-2)(m+1)$, is $\det(\text{Hess}(F))/(m-1)^{m+1}$. For a weighted projective space we also know that $J(F)$ has a socle degree $N=d(m+1)-\sum_{i=0}^m\deg(x_i)$ but we cannot apply the above result to say that a generator is the Hessian (since the generators of $J(F)$ do not have the same degree). In the case of any projective simplicial toric variety with Picard rank bigger than 1 one would like to understand the structure of the Jacobian ring of a quasi-smooth hypersurface (since it also describes its primitive Hodge structure). The main issue here is that the Jacobian ideal $J(F)=\langle \frac{\partial F}{\partial x_1},\ldots,\frac{\partial F}{\partial x_r}\rangle$ is not a complete intersection, i.e. is not generated by $m+1$ variables. For this reason we replace $J(F)$ by $J_0(F)=\langle x_1\frac{\partial F}{\partial x_1},\ldots,x_r\frac{\partial F}{\partial x_r}\rangle$ which in the non-degenerated case is a complete intersection of homogeneous polynomials of the same degree by \cref{propj0cox}.
\end{rmk}

\begin{dfn}
\label{deftorichess0}
Let $\P_\Sigma$ be an $m$-dimensional projective simplicial toric variety with anti-canonical class $\beta_0$ and let $X=\{F=0\}\subseteq\P_\Sigma$ be a non-degenerated hypersurface of degree $\beta\in\Cl(\P_\Sigma)$. Let $\rho_{i_1},\ldots,\rho_{i_m}$ be linearly independent primitive generators of the rays of $\Sigma$. We define the \textit{toric Hessian of $F$ with respect to $I=(i_1,\ldots,i_m)$} as
$$
\Hess ^I_\Sigma(F):=\Jac_\Sigma\left(F,x_{i_1}\frac{\partial F}{\partial x_{i_1}},\ldots,x_{i_m}\frac{\partial F}{\partial x_{i_m}}\right)\in S(\Sigma)^{(m+1)\beta-\beta_0}.
$$
\end{dfn}

\begin{cor}
\label{cortrace}
In the same context of \cref{deftorichess0}. For every $P\in R_0(F)^{(m+1)\beta-\beta_0}$ let $\xi_P^I\in H^m(\P_\Sigma,\widetilde{\Omega}_{\P_\Sigma}^m)$ be defined in Cech cohomology by the covering associated to $V(F,x_{i_1}\frac{\partial F}{\partial x_{{i_1}}},\ldots,x_{{i_m}}\frac{\partial F}{\partial x_{i_m}})=\varnothing$. Then 
$$
\text{Tr}(\xi^I_P)=(-1)^{m+1\choose 2} \cdot c\cdot m!\text{Vol}(\Delta),
$$
where $c\in\C$ is the unique number such that $$
P\equiv c\cdot \Hess^I_\Sigma(F) \ \ \text{(mod }J_0(F)),
$$
and $\Delta$ is the convex polyhedron associated to $\beta$.
\end{cor}

\begin{proof}
This follows directly from \cref{thmcox} and the fact that $(m+1)\beta-\beta_0$ is a socle degree for $J_0(F)$ by \cref{propunimodj0}.
\end{proof}

\begin{rmk}
It follows from \cref{thmcox} that for any two subsets $I,I'\subseteq\{1,\ldots,r\}$ of linearly independent rays of $\Sigma$, we have $\xi^I_{\Hess^I_\Sigma(F)}=\xi^{I'}_{\Hess^{I'}_\Sigma(F)}\in H^m(\P_\Sigma,\widetilde{\Omega}_{\P_\Sigma}^m)$ but this does not imply that $\Hess^I_\Sigma(F)$ and $\Hess^{I'}_\Sigma(F)$ are equal in $R_0(F)^{(m+1)\beta-\beta_0}$, since as Cech classes $\xi^I_{\Hess^I_\Sigma(F)}$ and $\xi^{I'}_{\Hess^{I'}_\Sigma(F)}$ are expressed in terms of different coverings. However, since $(m+1)\beta-\beta_0$ is a socle degree for $J_0(F)$, we know they coincide up to a scalar factor. Using the Euler relations (see \cref{propEulerrel}) it is possible to find explicitly this scalar factor. In other words we can normalize the Hessian. We will state the normalization of the Hessian leaving the proof for the reader. In any case, the proof of \cref{thm4} (in \S \ref{sec11}) also implies that this is the correct normalization (see \cref{corhess}).
\end{rmk}

\begin{dfn}
\label{deftorichess}
In the same context of \cref{deftorichess0} we define the \textit{toric Hessian of $F$} as
$$
\Hess_\Sigma(F):=\frac{\Hess_\Sigma^I(F)}{\det(\rho_I)}\in R_0(F)^{(m+1)\beta-\beta_0}.
$$
This is independent of the choice of $I$.
\end{dfn}

\section{Euler vector fields}
\label{sec7}
In this section we describe several identities between differential forms defined over $\C^r=\text{Spec }S(\Sigma)$ which are invariant under the action of $\mathbb{D}(\Sigma):=\text{Spec }\C[\Cl(\P_\Sigma)]$, for $\P_\Sigma$ a complete simplicial toric variety. These identities will be useful later to compute residues of meromorphic forms. These identities are mainly induced by the Euler vector fields (also called radial vector fields). Along this section we will fix a set of linearly independent primitive generators of the rays of $\Sigma$, and in order to simplify notation we will always assume this set to be
$$
\{\rho_{i_1},\ldots,\rho_{i_m}\}=\{\rho_1,\ldots,\rho_m\}.
$$

\begin{dfn}
\label{defEulervf}
Let $\P_\Sigma$ be an $m$-dimensional projective simplicial toric variety. For each $j\in\{m+1,\ldots,r\}$ we have an \textit{Euler vector field}
$$
E_j:=x_j\frac{\partial}{\partial x_j}-\sum_{i=1}^ma_{i,j}x_i\frac{\partial}{\partial x_i}
$$
where $\rho_j=\sum_{i=1}^ma_{i,j}\rho_i$.
\end{dfn}

\begin{rmk}
\label{rmkdifformsE}
The main reason we need to work with Euler vector fields, is because they determine when a $\mathbb{D}(\Sigma)$-invariant differential form $\eta$ of $\C^r$ descends to a differential form on $\P_\Sigma$. In fact, this happens if and only if it vanishes under all contractions 
$$
\iota_{E_j}(\eta)=0 \ \ \text{ for }j=m+1,\ldots,r.
$$
In consequence, if $m=\dim\P_\Sigma$, $\beta\in\Cl(\P_\Sigma)$ is any Cartier class and $\eta\in H^0(\P_\Sigma,\Omega_{\P_\Sigma}^m(\beta))$, then we can write
$$
\eta=\sum_{|I|=m}f_I dx_I
$$
where each $f_I\in S(\Sigma)$ is homogeneous and all Euler contractions vanish
$$
0=\iota_{E_j}(\eta)=\sum_{|I|=m}\left(f_Ix_j\iota_{\frac{\partial}{\partial x_j}}(dx_I)-\sum_{i=1}^mf_Ia_{i,j}x_i\iota_{\frac{\partial}{\partial x_i}}(dx_I)\right)
$$
$$
=\sum_{|J|=m-1}\left(f_{(j,J)}x_j-\sum_{i=1}^mf_{(i,J)}a_{i,j}x_i\right)dx_J
$$
In other words we have the following linear relations between elements of $S(\Sigma)^{\beta+\beta_0}$ 
\begin{equation}
\label{eq2}
f_{(j,J)}x_{J\cup\{j\}}=\sum_{i=1}^mf_{(i,J)}a_{i,j}x_{J\cup\{i\}}    
\end{equation}
for all $|J|=m-1$ and $j\in\{m+1,\ldots,r\}$. On the other hand, it is an elementary linear algebra fact that given any function $c:\bigwedge^m\C^r\rightarrow \C$ such that 
$$
c(e_j\wedge e_{j_1}\wedge\cdots\wedge e_{j_{m-1}})=\sum_{i=1}^mc(e_i\wedge e_{j_1}\wedge\cdots\wedge e_{j_{m-1}})\cdot a_{i,j} 
$$
for all $j\in\{m+1,\ldots,r\}$ and all $J=(j_1,\ldots,j_{m-1})$, then there exists some constant $C\in\C$ such that
$$
c(e_{j_1}\wedge\cdots\wedge e_{j_m})=C\cdot \det(A_{\bullet j_1}|\cdots|A_{\bullet j_m})
$$
where $A=(Id_{m\times m}|a_{\bullet m+1}|\cdots|a_{\bullet r})$. Since $(A_{\bullet j_1}|\cdots|A_{\bullet j_m})=(\rho_1|\cdots|\rho_m)^{-1}\cdot(\rho_J)$, we conclude from \eqref{eq2} that there exists some polynomial $Q\in S(\Sigma)^{\beta+\beta_0}$ such that 
$$
f_Ix_I=Q\cdot\det(\rho_I)
$$
for all $|I|=m$. This implies that $Q=x_1\cdots x_r\cdot P$ for some $P\in S(\Sigma)^\beta$ and so $\eta=P\cdot \Omega$. In other words we have the isomorphism
$$
H^0(\P_\Sigma,\mathcal{O}(\beta))\simeq H^0(\P_\Sigma,\widetilde{\Omega}_{\P_\Sigma}^m(\beta)).
$$
\end{rmk}

Euler vector fields induce generalized Euler relations introduced first in \cite[Lemma 3.8]{batyrev1994hodge} which we recall as follows.

\begin{prop}
\label{propEulerrel}
In the same context of \cref{defEulervf}, let $F\in S(\Sigma)^\beta$. Write $\beta=\sum_{i=1}^rb_i\cdot D_i$, where each $D_i$ is the toric divisor corresponding to the ray generated by $\rho_i$. For each $j\in\{m+1,\ldots,r\}$ consider $k_j:=b_j-\sum_{i=1}^mb_i\cdot a_{i,j}$. Then the \textit{Euler relations} hold
\begin{equation}
\label{eqEulerrel}
\iota_{E_j}(dF)=k_j\cdot F.
\end{equation}
\end{prop}

\begin{proof}
This is \cite[Lemma 3.8]{batyrev1994hodge} applied to the relation $\rho_j-\sum_{i=1}^ma_{i,j}\rho_i=0$.
\end{proof}

\begin{prop}
In the same context of \cref{defEulervf} we have
\begin{equation}
\label{eqomegaEuler}
\Omega=(-1)^{m\cdot (r-1)}\det(\rho_1|\cdots|\rho_m)\cdot \iota_{E_r}(\cdots\iota_{E_{m+1}}(V)\cdots)    
\end{equation}
where $\iota_v$ denotes the contraction operator with respect to a vector field $v$ and $$V:=dx_1\wedge dx_2\wedge\cdots\wedge dx_r.$$ 
\end{prop}

\begin{proof}
This follows by induction on $r$. In fact, if $r=m+1$
$$
\det(\rho_1|\cdots|\rho_m)\cdot\iota_{E_{m+1}}(dx_1\wedge\cdots\wedge dx_{m+1})=(-1)^{m}\Omega
$$
since $(\rho_1|\cdots|\rho_{m+1})=(\rho_1|\cdots|\rho_m)\cdot(Id_{m\times m}|a_{\bullet m+1})$. If $r>m+1$ we have by induction that
$$
\det(\rho_1|\cdots|\rho_m)\cdot\iota_{E_{r-1}}(\cdots\iota_{E_{m+1}}(dx_1\wedge\cdots\wedge dx_{r-1})\cdots)=(-1)^{m\cdot(r-2)}\Omega'
$$
where
$$
\Omega'=\sum_{|I|=m, \ r\notin I}\det(\rho_I)\widehat{x_{I\cup\{r\}}}dx_I.
$$
Then
$$
\det(\rho_1|\cdots|\rho_m)\cdot\iota_{E_{r}}(\cdots\iota_{E_{m+1}}(dx_1\wedge\cdots\wedge dx_{r})\cdots)=(-1)^{m\cdot(r-2)}\iota_{E_r}(\Omega'\wedge dx_r)
$$
$$
=(-1)^{m\cdot(r-2)}\sum_{|I|=m, \ r\notin I}\det(\rho_I)\left((-1)^m\widehat{x_I}dx_I-\sum_{i=1}^ma_{i,r}\widehat{x_{I\cup\{r\}}}x_i\iota_{x_i}(dx_I)\wedge dx_r\right)
$$
$$
=(-1)^{m\cdot(r-1)}\left(\sum_{|I|=m,\ r\notin I}\det(\rho_I)\widehat{x_I}dx_I\right.
$$
$$
+\left.\sum_{|J|=m-1, \ r\notin J}\left(\sum_{i=1}^m(-1)^{m+i}a_{i,r}\det(\rho_{J\cup\{i\}})\right)\widehat{x_{J\cup\{r\}}}dx_{J\cup\{r\}}\right).
$$
The result follows after noting that $\rho_r=\sum_{i=1}^ma_{i,r}\rho_i$ implies $$\det(\rho_{J\cup \{r\}})=\sum_{i=1}^m(-1)^{m+i}a_{i,r}\det(\rho_{J\cup\{i\}}).$$
\end{proof}

\begin{cor}
\label{corVbeta1}
In the same context of \cref{propEulerrel}, we have the following identity
\begin{equation}
\label{eqdFomega}    
\Omega\wedge dF=(-1)^{(m-1)r}\cdot F\cdot V^\beta
\end{equation}
where $V^\beta$ can be computed with respect to $\sigma\in\Sigma(m)$ (whose rays are generated by $\rho_1,\ldots,\rho_m$) as follows
\begin{equation}
\label{eqVbeta}
V^\beta=\det(\rho_1|\cdots|\rho_m)\sum_{j=m+1}^r(-1)^{j}k_j\widehat{\iota_{E_j}}(V)    
\end{equation}
and $\widehat{\iota_{E_j}}:=\iota_{E_r}\cdots \widehat{\iota_{E_j}}\cdots \iota_{E_{m+1}}$.
\end{cor}

\begin{proof}
Apply the vector field contractions $\iota_{E_r}\cdots \iota_{E_{m+1}}$ to the identity $V\wedge dF=0$ and use \eqref{eqomegaEuler} and \eqref{eqEulerrel}.    
\end{proof}

\begin{cor}
In the same context of the previous corollary we have for $J=\{j_0<\cdots<j_l\}\subseteq\{1,\ldots,r\}$ that
\begin{equation}
\label{eqdFOmegacontr}
\Omega_J\wedge dF+(-1)^m\sum_{k=0}^l(-1)^kF_{j_k}\Omega_{J\setminus\{j_k\}}=(-1)^{(m-1)r}\cdot F\cdot V^\beta_J
\end{equation}
where for any differential form $\eta$ we denote $\eta_J:=\iota_{\frac{\partial}{\partial x_{j_l}}}(\cdots\iota_{\frac{\partial}{\partial x_{j_0}}}(\eta)\cdots)$.
\end{cor}

\begin{proof}
This follows after applying the vector field contractions $\iota_{\frac{\partial}{\partial x_{j_l}}}\cdots\iota_{\frac{\partial}{\partial x_{j_0}}}$ to the identity \eqref{eqdFomega}.    
\end{proof}

\begin{rmk}
In the same context of \cref{corVbeta1}, it follows directly from \eqref{eqVbeta} and \eqref{eqomegaEuler} that
\begin{equation}
\iota_{E_j}(V^\beta)=(-1)^{(m-1)(r-1)+1}k_j\Omega.
\end{equation}
In particular for $J=\{j_0<\cdots<j_l\}\subseteq\{1,\ldots,r\}$
\begin{equation}
\label{eqVbetajotaEuler}
\iota_{E_j}(V^\beta_J)=(-1)^{(m-1)(r-1)+l}k_j\Omega_J.
\end{equation}
\end{rmk}

\section{Toric Carlson-Griffiths theorem}
\label{sec8}
In this section we generalize some results of \cite{carlson1980infinitesimal} to the toric case. As we saw in \S \ref{sec5}, the primitive cohomology of a quasi-smooth intersection inside a projective simplicial toric variety can be described in terms of residue forms. The main result of \cite{carlson1980infinitesimal} we generalize here is about the computation of these residues in Cech cohomology. We begin with the following main lemma.

\begin{dfn}
\label{deflemaop}
Let $X$ be an $n$-dimensional quasi-smooth subvariety of a projective simplicial toric variety $\P_\Sigma$, and $Y:=X\cap \{F=0\}$ be a quasi-smooth hypersurface of $X$ given by some homogeneous polynomial $F\in S(\Sigma)=\C[x_1,\ldots,x_r]$ of ample degree. Consider $\U=\{U_i\}_{i=1}^r$ the Jacobian covering of $X$ given by $U_i=X\cap \{F_i\neq 0\}$, where $F_i:=\frac{\partial F}{\partial x_i}$. For every $l\ge 2$ define 
$$
H^{p,q}_l:C^q(\U,\widetilde\Omega_X^p(lY))\rightarrow C^q(\U, \widetilde\Omega_X^{p-1}((l-1)Y)),
$$
$$
(H^{p,q}_l\omega)_{j_0\cdots j_q}:=\frac{(-1)^{q}}{1-l}\frac{F}{F_{j_0}}\iota_{\frac{\partial}{\partial x_{j_0}}}(\omega_{j_0\cdots j_q})
$$
and 
$$
H_l:=\bigoplus_{p+q=n} H^{p,q}_l:\bigoplus_{p+q=n}C^q(\U, \widetilde\Omega_X^{p}(lY))\rightarrow \bigoplus_{p+q=n}C^q(\U, \widetilde\Omega_X^{p-1}((l-1)Y)).
$$
\end{dfn}

\begin{lemma}[Toric Carlson-Griffiths Lemma]
\label{lemmaCG}
For every $l\ge 2$, letting $D=\delta+(-1)^qd$, then 
$$
DH_l+H_{l+1}D:\bigoplus_{p+q=n}\frac{C^q(\U,\widetilde\Omega_{X}^p(lY))}{C^q(\U,\widetilde\Omega_{X}^p((l-1)Y))}\rightarrow \bigoplus_{p+q=n}\frac{C^q(\U,\widetilde\Omega_{X}^p(lY))}{C^q(\U,\widetilde\Omega_{X}^p((l-1)Y))} $$
is the identity map. Note that $d$ is the usual exterior differential, while $\delta$ is the Cech differential.
\end{lemma}

\begin{proof}
The same proof of \cite{carlson1980infinitesimal} applies. See for instance \cite[Theorem 6.2]{movasati2020course}.
\end{proof}

Using the above lemma, we can compute residue forms along ample quasi-smooth hypersurfaces of projective simplicial toric varieties. 

\begin{thm}[Toric Carlson-Griffiths Theorem] 
\label{thmtoricCG}
Let $\P_\Sigma$ be an $m$-dimensional projective simplicial toric variety with anti-canonical class $\beta_0\in\Cl(\P_\Sigma)$. Let $X=\{F=0\}\subseteq\P_\Sigma$ be an ample quasi-smooth hypersurface of ample degree $\beta=\deg(F)\in \Cl(\P_\Sigma)$. Let $q\in \{0,1,\ldots,m-1\}$ and $P\in S(\Sigma)^{\beta(q+1)-\beta_0}$, then 
\begin{equation}
\label{5.9.2018.2}
\res \left(\frac{P\Omega}{F^{q+1}}\right)^{m-1-q,q}=\frac{(-1)^{m-1}}{q!}\left\{\frac{P\Omega_J}{F_J}\right\}_{|J|=q+1}\in H^q(\U,\widetilde{\Omega}_{X}^{m-1-q}).
\end{equation}
Where $\Omega_J:=\iota_{\frac{\partial}{\partial x_{j_q}}}(\cdots\iota_{\frac{\partial}{\partial x_{j_0}}}(\Omega)\cdots)$, $F_J:=F_{j_0}\cdots F_{j_q}$ and $\U=\{U_i\}_{i=1}^{r}$ is the Jacobian covering restricted to $X$, given by $U_i=\{F_i\neq 0\}\cap X$. 
\end{thm}

\begin{proof}
Let $U:=\P_\Sigma\setminus X$. In order to compute the residue of $\omega:=\frac{P\Omega}{F^{q+1}}\in \Hip^m(U,\widetilde{\Omega}^\bullet_U)$ we use the Carlson-Griffiths Lemma (\cref{lemmaCG}) to reduce the pole order as follows: For $l=0,\ldots,q$ define 
 $$
 ^{(l)}\omega:=(1-DH_{q+2-l})\cdots(1-DH_q)(1-DH_{q+1})\left(\frac{P\Omega}{F^{q+1}}\right)\in \Hip^{m}(\mathcal{U},\widetilde{\Omega}^\bullet_U),
 $$
 where $H_l=\bigoplus_{p+q=m}H^{p,q}_l$ is the operator defined in \cref{deflemaop}. We claim 
 $$
 {}^{(l)}\omega^{m-l}=\left\{\frac{(q-l)!(-1)^{m-1}P}{q!\cdot F^{q-l}}\left(\frac{\Omega_J}{F_J}\wedge\frac{dF}{F}-(-1)^{(m-1)r}\frac{V^\beta_J}{F_J}\right)\right\}_{|J|=l+1}
 $$
 as an element of $C^l(\U, \widetilde{\Omega}_{\P_\Sigma}^{m-l}((q-l+1)X))$.
 In fact, for $l=0$, the claim follows from identity \eqref{eqdFOmegacontr}. Assuming the claim for $l\ge 0$, then 
 $$
 H^{m-l,l}_{q-l+1}({}^{(l)}\omega^{m-l})_J=\frac{-(q-l-1)!P}{q!\cdot F^{q-l}}\frac{\Omega_J}{F_J}.
 $$
 In consequence, 
 $$
 {}^{(l+1)}\omega_J^{m-l-1}=-\delta H^{m-l,l}_{q-l+1}({}^{(l)}\omega^{m-l})_J=\frac{(q-l-1)!}{q!}\sum_{k=0}^{l+1}(-1)^{k}\frac{P\Omega_{J\setminus\{j_k\}}}{F^{q-l}F_{J\setminus\{j_k\}}}.
 $$
Using identity \eqref{eqdFOmegacontr} we obtain the claim for $l+1$. In conclusion 
$$
 {}^{(q)}\omega^{m-q}=\left\{\frac{(-1)^{m-1}P}{q!}\left(\frac{\Omega_J}{F_J}\wedge\frac{dF}{F}-(-1)^{(m-1)r}\frac{V^\beta_J}{F_J}\right)\right\}_{|J|=q+1}
 $$
is an element of $C^q(\U, \widetilde{\Omega}_{\P_\Sigma}^{m-q}(\log X))$, thus applying the residue map we get the result.
\end{proof}

\section{Differential forms associated to divisor classes}
\label{sec9}
In this section we generalize the differential form $V^\beta$ introduced in \cref{corVbeta1} to an arbitrary amount of divisor classes. Then we describe some relations between these differential forms. We will see in the next section that these forms arise naturally when we compute residues of logarithmic forms. As in \S \ref{sec7} we will fix a set of linearly independent primitive generators of the rays of $\Sigma$, and we will always assume this set to be
$$
\{\rho_{i_1},\ldots,\rho_{i_m}\}=\{\rho_1,\ldots,\rho_m\}.
$$

\begin{dfn}
\label{defValphais}
Let $\P_\Sigma$ be an $m$-dimensional projective simplicial toric variety. Let $\alpha_1,\ldots,\alpha_p\in\Cl(\P_\Sigma)$ given by $\alpha_\ell=\sum_{i=1}^rb_{\ell, i}\cdot D_i$, where each $D_i$ is the toric divisor corresponding to the ray generated by $\rho_i$. For each $j\in\{m+1,\ldots,r\}$ and each $\ell\in\{1,\ldots,p\}$ define $k_{\ell, j}:=b_{\ell, j}-\sum_{i=1}^mb_{\ell, i}\cdot a_{i,j}$, where $\rho_j=\sum_{i=1}^ma_{i,j}\cdot\rho_i$. Then we define
\begin{equation}
V^{\alpha_1,\ldots,\alpha_p}:=\det(\rho_1|\cdots|\rho_m)\sum_{m+1\le j_1<\cdots<j_p\le r}(-1)^{j_1+\cdots+j_p}\widehat{\iota_{E_{j_1,\ldots,j_p}}}(V)\det\begin{pmatrix}k_{1,j_1} & \cdots & k_{p,j_1} \\
\vdots & \ddots & \vdots \\ 
k_{1,j_p} & \cdots 
& k_{p,j_p}\end{pmatrix}    
\end{equation}
where $\widehat{\iota_{E_{j_1,\ldots,j_p}}}:=\iota_{E_r}\cdots\widehat{\iota_{E_{j_p}}}\cdots\widehat{\iota_{E_{j_{p-1}}}}\cdots\cdots\widehat{\iota_{E_{j_1}}}\cdots\iota_{E_{m+1}}$.
\end{dfn}

A direct consequence of the following proposition is that the form $V^{\alpha_1,\ldots,\alpha_p}$ does not depend on the choice of the fixed linearly independent primitive generators of the rays.

\begin{prop}
In the same context of \cref{defValphais}, let $\beta\in\Cl(\P_\Sigma)$ be an ample class and $F\in S(\Sigma)^\beta$. Then 
\begin{equation}
V^{\alpha_1,\ldots,\alpha_p}\wedge dF=(-1)^{r+m+p}\cdot F\cdot V^{\alpha_1,\ldots,\alpha_p,\beta}.
\end{equation}
In consequence we have for $J=\{j_0<\cdots<j_l\}\subseteq\{1,\ldots,r\}$ that
\begin{equation}
\label{eqValfasJdF}
V^{\alpha_1,\ldots,\alpha_p}_J\wedge dF+(-1)^{m+p}\sum_{k=0}^l(-1)^kF_{j_k}V^{\alpha_1,\ldots,\alpha_p}_{J\setminus\{j_k\}}=(-1)^{r+m+p}\cdot F\cdot V^{\alpha_1,\ldots,\alpha_p,\beta}_J.
\end{equation}
\end{prop}

\begin{proof}
Noting that $V\wedge dF=0$, if we apply the contraction $\widehat{\iota_{E_{j_1,\ldots,j_p}}}$ to this equality we get
$$
0=\widehat{\iota_{E_{j_1,\ldots,j_p}}}(V)\wedge dF+(-1)^{r+m+1}\left(\sum_{\ell_0<j_1}(-1)^{\ell_0} \widehat{\iota_{E_{\ell_0,j_1,\ldots,j_p}}}(V)\cdot F\cdot k_{\ell_0}+\right.
$$
$$
\left.-\sum_{j_1<\ell_1<j_2}(-1)^{\ell_1} \widehat{\iota_{E_{j_1,\ell_1,j_2,\ldots,j_p}}}(V)\cdot F\cdot k_{\ell_1}+\cdots+(-1)^{p}\sum_{j_p<\ell_p}(-1)^{\ell_p} \widehat{\iota_{E_{j_1,\ldots,j_p,\ell_p}}}(V)\cdot F\cdot k_{\ell_p}\right).
$$
Therefore if follows that
$$
0=(-1)^{r+m+p}\cdot V^{\alpha_1,\ldots,\alpha_p}\wedge dF+
$$
$$
\det(\rho_1|\cdots|\rho_m)\sum_{j_1<\cdots<j_{p}<\ell}(-1)^{j_1+\cdots+j_p+\ell}\cdot F\cdot\widehat{\iota_{E_{j_1,\ldots,j_p,\ell}}}(V)\det\begin{pmatrix}
k_{j_1} & k_{1,j_1} & \cdots & k_{p,j_1} \\
\vdots & \vdots & \ddots & \vdots \\
k_{j_p} & k_{1,j_p} & \cdots & k_{p,j_p} \\
k_\ell & k_{1,\ell} &\cdots & k_{p,\ell}
\end{pmatrix}
$$
as desired.
\end{proof}

The following proposition tells us how do these forms relate via Euler contractions.

\begin{rmk}
In the same context of \cref{defValphais} we have just by definition that for $p>1$
\begin{equation}
\iota_{E_j}(V^{\alpha_1,\ldots,\alpha_p})=(-1)^{r+p}\sum_{i=1}^p(-1)^i k_{i, j}\cdot V^{\alpha_1,\ldots,\widehat{\alpha_i},\ldots,\alpha_p}.
\end{equation}
In consequence for $J=\{j_0<\cdots<j_l\}\subseteq\{1,\ldots,r\}$ we have that
\begin{equation}
\label{eqValfasJE}
\iota_{E_j}(V^{\alpha_1,\ldots,\alpha_p}_J)=(-1)^{r+p+l+1}\sum_{i=1}^p(-1)^i k_{i, j}\cdot V^{\alpha_1,\ldots,\widehat{\alpha_i},\ldots,\alpha_p}_J.
\end{equation}
\end{rmk}

\section{Poincar\'e residue}
\label{sec10}
In this section we recall the Poincar\'e residue map and use it to compute the successive residues of top forms of quasi-smooth complete intersection subvarieties of simplicial projective toric varieties.

Given $X$ a quasi-smooth $n$-dimensional projective orbifold and $Y\subseteq X$ a quasi-smooth ample divisor, there exists a \textit{Poincar\'e residue morphism} of coherent sheaves over $X$
$$
res: \widetilde{\Omega}_X^k(\log Y)\rightarrow i_*\widetilde{\Omega}_Y^{k-1}
$$
which locally takes a logarithmic form $\alpha\wedge \frac{df}{f}+\beta$ to $\alpha|_Y$ (where $Y$ is locally given by $\{f=0\}$ and $i:Y\hookrightarrow X$ is the inclusion). This morphism fits in a short exact sequence called the \textit{Poincar\'e residue sequence}
$$
0\rightarrow \widetilde{\Omega}_X^k\rightarrow\widetilde{\Omega}_X^k(\log Y)\xrightarrow{res} i_*\widetilde{\Omega}_Y^{k-1}\rightarrow 0.
$$
Taking the long exact sequence associated in cohomology one gets for $k=n$ that the coboundary map
\begin{equation}
\label{coboundary}
\delta: H^{n-1}(Y,\widetilde{\Omega}_Y^{n-1})\xrightarrow{\sim} H^{n}(X,\widetilde{\Omega}_X^{n})    
\end{equation}
is an isomorphism. Moreover, since this map corresponds to the cup with the polarization of $X$, it follows that $\delta$ commutes with the trace maps. Thus we can use it to translate the trace of a top form over $Y$ to the trace of a top form over $X$. This is particularly useful in the following context: Let $X\subseteq \P_\Sigma$ be an $n$-dimensional quasi-smooth hypersurface of a projective simplicial toric variety $\P_\Sigma$ of odd dimension $n+1$. Let 
$$
Z=\{g_1=\cdots=g_{\frac{n}{2}+1}=0\}\subseteq\P_\Sigma
$$
be a complete intersection of ample divisors of $\P_\Sigma$ such that $Z\subseteq X$. For each $i=0,\ldots,\frac{n}{2}+1$ define
$$
Z_i:=\{g_{i+1}=\cdots=g_{\frac{n}{2}+1}=0\}\subseteq\P_\Sigma.
$$

\begin{prop}
\label{propPoincres}
Suppose that each $Z_i$ is quasi-smooth. Let $X=\{F=0\}$ for $F\in S(\Sigma)^\beta$ and denote $r=\#\Sigma(1)$. Write $F=g_1\cdot h_1+\cdots+ g_{\frac{n}{2}+1}\cdot h_{\frac{n}{2}+1}$. Then, for every $P\in R(F)^{(\frac{n}{2}+1)\beta-\beta_0}$ and every $\ell=0,\ldots,\frac{n}{2}+1$ we have $\delta^{\ell}(\omega_P|_Z)\in H^{\frac{n}{2}+\ell}(Z_{\ell},\widetilde{\Omega}_{Z_\ell}^{\frac{n}{2}+\ell})$ is 
$$
\delta^{\ell}(\omega_P|_Z)=\left\{\frac{P}{\frac{n}{2}!F_J}\sum_{k=0}^{\ell+1}\sum_{\begin{smallmatrix}|I|=k \\ I\subseteq [2\ell]\end{smallmatrix}}(-1)^{r(k-1)+{k+2\choose 2}+\ell}V_J^{\deg(e_I)}\wedge e^*((dz_0\wedge\cdots\wedge dz_{2\ell-1})_I)\right\}_{|J|=\frac{n}{2}+\ell+1}
$$
where $[2\ell]:=\{0,1,\ldots,2\ell-1\}$, $e:\C^{r}\rightarrow \C^{n+2}$ is given by $$(e_0,\ldots,e_{n+1})=(g_1,h_1,\ldots,g_{\frac{n}{2}+1},h_{\frac{n}{2}+1})$$ and we denote $\deg(e_I)=(\deg(e_{i_1}),\ldots,\deg(e_{i_k}))$.
\end{prop}

\begin{proof}
Let us proceed inductively on $\ell$. For $\ell=0$ we have
$$
\delta^0(\omega_P|_Z)=\omega_P|_Z=\left\{\frac{P}{\frac{n}{2}!F_J}\Omega_J\right\}_{|J|=\frac{n}{2}+1}=\left\{\frac{P}{\frac{n}{2}!F_J}(-1)^{r+1}V_J^\varnothing\right\}_{|J|=\frac{n}{2}+1}\in H^{\frac{n}{2}}(Z_0,\widetilde{\Omega}_{Z_0}^{\frac{n}{2}}).
$$
For $\ell>0$ assume $\delta^{\ell-1}(\omega_P|_Z)$ satisfies the formula. In order to compute $\delta^\ell(\omega_P|_Z)$ we need to find for each $|J|=\frac{n}{2}+\ell+1$ some $\eta\in \widetilde{\Omega}_{Z_{\ell+1}}^{\frac{n}{2}+\ell+1}(\log Z_\ell)(U_J)$ such that $res(\eta)=\omega^{(\ell)}$ where
$$
\omega^{(\ell)}:=\frac{P}{\frac{n}{2}!F_J}\sum_{k=0}^{\ell+1}\sum_{\begin{smallmatrix}|I|=k \\ I\subseteq [2\ell]\end{smallmatrix}}(-1)^{r(k-1)+{k+2\choose 2}+\ell}V_J^{\deg(e_I)}\wedge e^*((dz_0\wedge\cdots\wedge dz_{2\ell-1})_I).
$$
We claim that 
$$
\eta=\omega^{(\ell)}\wedge \frac{dg_{\ell+1}}{g_{\ell+1}}+\widetilde{\omega}^{(\ell)}
$$
where 
$$
\widetilde{\omega}^{(\ell)}:=\frac{P}{\frac{n}{2}!F_J}\sum_{k=0}^{\ell+1}\sum_{\begin{smallmatrix}|I|=k \\ I\subseteq [2\ell]\end{smallmatrix}}(-1)^{rk+{k+2\choose 2}+\ell}V_J^{\deg(e_I),\deg(g_{\ell+1})}\wedge e^*((dz_0\wedge\cdots\wedge dz_{2\ell-1})_I).
$$
In fact, using the identities \eqref{eqEulerrel} and \eqref{eqValfasJE} it is not hard to see that $\iota_{E_j}(\omega^{(\ell)}\wedge \frac{dg_{\ell+1}}{g_{\ell+1}}+\widetilde{\omega}^{(\ell)})=0$ for all $j=1,\ldots,r$ and so it is the desired $\eta$ by \cref{rmkdifformsE}. Finally taking the Cech differential and using \eqref{eqValfasJdF} we compute $\delta^\ell(\omega_P|Z)$ and check it also satisfies the formula.
\end{proof}

\begin{cor}
\label{coriteratedres}
In the same context of \cref{propPoincres}, let $\rho_{i_1},\ldots,\rho_{i_{n+1}}$ be linearly independent primitive generators of the rays of $\Sigma$, then
$$
\delta^{\frac{n}{2}+1}(\omega_P|_Z)=\frac{sgn(I)(-1)^{\frac{n}{2}}P\cdot \det\begin{pmatrix}
k_{0,\ell}e_0 & \cdots & k_{n+1, \ell}e_{n+1} \\ \frac{\partial e_0}{\partial x_{i_1}} & \cdots & \frac{\partial e_{n+1}}{\partial x_{i_1}} \\ \vdots &  & \vdots \\ \frac{\partial e_0}{\partial x_{i_{n+1}}} & \cdots & \frac{\partial e_{n+1}}{\partial x_{i_{n+1}}} 
\end{pmatrix}}{\frac{n}{2}!k_\ell\det(\rho_I)F\cdot F_{i_1}\cdots F_{i_{n+1}}}\cdot\Omega\in H^{n+1}(\mathcal{U}^I,\widetilde{\Omega}_{\P_\Sigma}^{n+1}),
$$
where $\mathcal{U}^I$ is the open covering associated to $V(F,x_{i_1}\frac{\partial F}{\partial x_{i_1}},\ldots,x_{i_{n+1}}\frac{\partial F}{\partial x_{i_{n+1}}})=\varnothing\subseteq\P_\Sigma$, and $\ell\in\{1,\ldots,r\}\setminus\{i_1,\ldots,i_{n+1}\}$ is any number such that $k_\ell\neq 0$. 
\end{cor}

\begin{proof}
By \cref{rmkdifformsE} we can write in the Jacobian covering $\mathcal{U}$ 
$$
\delta^{\frac{n}{2}+1}(\omega_P|_Z)=\left\{\frac{P\cdot R^J\cdot \Omega}{\frac{n}{2}!\cdot x_{j_0}F_{j_0}\cdots x_{j_{n+1}}F_{j_{n+1}}}\right\}_{|J|=n+2}\in H^{n+1}(\mathcal{U},\widetilde{\Omega}_{\P_\Sigma}^{n+1})
$$ 
for some $R^J\in S(\Sigma)^{(\frac{n}{2}+1)\beta}$. In order to compute its image in $H^{n+1}(\mathcal{U}^I,\widetilde{\Omega}_{\P_\Sigma}^{n+1})$ it is enough to pass to a common refinement $\mathcal{V}=\{V_i\}_{i=0}^r$ where $V_0:=\{F\neq 0\}$ and $V_i:=\{x_i\frac{\partial F}{\partial x_i}\neq 0\}$ for each $i=1,\ldots,r$. Now, for computing $\delta^{\frac{n}{2}+1}(\omega_P|_Z)$ in the refinement it is enough to fix any of non-zero Euler relation
$$
1=\frac{1}{k_\ell F}\left(x_{\ell}\frac{\partial F}{\partial x_{\ell}}-\sum_{j=1}^{n+1}a_{i_j,\ell}x_{i_j}\frac{\partial F}{\partial x_{i_j}}\right)
$$
and so
$$
\delta^{\frac{n}{2}+1}(\omega_P|_Z)_{0,i_1,\ldots,i_{n+1}}=\frac{P\Omega}{\frac{n}{2}!k_\ell F\cdot x_{i_1}F_{i_1}\cdots x_{i_{n+1}}F_{i_{n+1}}}\left(R^{\ell,i_1,\ldots,i_{n+1}}-\sum_{j=1}^{n+1}a_{i_j,\ell}R^{i_j,i_1,\ldots,i_{n+1}}\right)
$$
$$
=\frac{P\cdot R^{\ell,i_1,\ldots,i_{n+1}}\cdot \Omega}{\frac{n}{2}!k_\ell F\cdot x_{i_1}F_{i_1}\cdots x_{i_{n+1}}F_{i_{n+1}}}.
$$
The result follows applying \cref{propPoincres} to compute $R^{\ell,i_1,\ldots,i_{n+1}}$.
\end{proof}

\section{Proof of the main theorems}
\label{sec11}
Let us star by  showing the following proposition which generalizes \cite[Theorem 2]{carlson1980infinitesimal}.

\begin{prop}
\label{propcup}
Let $\P_\Sigma$ be a projective simplicial toric variety with anti-canonical class $\beta_0\in\Cl(\P_\Sigma)$. Let $X=\{F=0\}\subseteq\P_\Sigma$ be an ample quasi-smooth hypersurface of degree $\beta=\deg(F)\in \Cl(\P_\Sigma)$ and even dimension $n$. Let $P,Q\in S(\Sigma)^{\beta(\frac{n}{2}+1)-\beta_0}$. For any set 
 $\rho_{i_1},\ldots,\rho_{i_{n+1}}$ of linearly independent primitive generators of the rays of $\Sigma$ we have that
\begin{equation}
\delta(\omega_P\cup\omega_Q)=\frac{(-1)^{\frac{n}{2}}PQ\widehat{x_{i_1,\ldots,i_{n+1}}}\det(\rho_I)}{(\frac{n}{2}!)^2F\cdot F_{i_1}\cdots F_{i_{n+1}}}\Omega\in H^{n+1}(\U^I,\widetilde{\Omega}_{\P_\Sigma}^{n+1}),
\end{equation}
where $\delta: H^n(X,\widetilde{\Omega}_X^n)\xrightarrow{\sim} H^{n+1}(\P_\Sigma,\widetilde{\Omega}_{\P_\Sigma}^{n+1})$ is the map \eqref{coboundary} and $\U^I$ is the covering associated to $V(F,x_{i_1}\frac{\partial F}{\partial x_{i_1}},\ldots,x_{i_{n+1}}\frac{\partial F}{\partial x_{i_{n+1}}})=\varnothing\subseteq\P_\Sigma$.
\end{prop}

\begin{proof}
Without loss of generality we can assume that $I=(1,2,\ldots,n+1)$. Using the toric Carlson-Griffiths theorem \cref{thmtoricCG} and the formula for the cup product in Cech cohomology we get that
$$
\omega_P\cup \omega_Q=\left\{\frac{(-1)^\frac{n}{2}PQ\Omega_{j_0\ldots j_\frac{n}{2}}\wedge \Omega_{j_{\frac{n}{2}}\ldots j_{n}}}{(\frac{n}{2}!)^2F_J\cdot F_{j_\frac{n}{2}}}\right\}_{|J|=n+1}\in H^n(\mathcal{U},\widetilde{\Omega}_X^n),
$$
where $\mathcal{U}$ is the Jacobian covering. Since $\Omega_{j_0\ldots j_\frac{n}{2}}\wedge \Omega_{j_{\frac{n}{2}}\ldots j_{n}}=\det(\rho_J)\widehat{x_{J}}\Omega_{j_\frac{n}{2}}$ it follows that
$$
\omega_P\cup \omega_Q=\left\{\frac{(-1)^\frac{n}{2}PQ\det(\rho_1|\cdots|\rho_{n+1})\widehat{x_{1,\ldots,{n+1}}}\Omega_{j_\frac{n}{2}}}{(\frac{n}{2}!)^2F_J\cdot F_{j_\frac{n}{2}}}\right\}_{|J|=n+1}\in H^n(\mathcal{U},\widetilde{\Omega}_X^n).
$$
By \eqref{eqVbetajotaEuler} it is easy to see that the form 
$$
\eta:=\frac{(-1)^\frac{n}{2}PQ\det(\rho_J)\widehat{x_{J}}\Omega_{j_\frac{n}{2}}}{(\frac{n}{2}!)^2F_J\cdot F_{j_\frac{n}{2}}}\wedge\frac{dF}{F}+\frac{(-1)^{\frac{n}{2}+1}PQ\det(\rho_J)\widehat{x_{J}}V^\beta_{j_\frac{n}{2}}}{(\frac{n}{2}!)^2F_J\cdot F_{j_\frac{n}{2}}}
$$
is such that $\res(\eta)=(\omega_P\cup\omega_Q)_J$. Using this we can compute 
$$
\delta(\omega_P\cup\omega_Q)=\left\{\frac{(-1)^\frac{n}{2}\widehat{x_J}PQ\Omega}{(\frac{n}{2}!)^2F_J\cdot F}\sum_{i=0}^{n+1}(-1)^i\det(\rho_{J\setminus\{j_i\}})x_{j_i}F_{{j_i}}\right\}_{|J|=n+2}\in H^{n+1}(\mathcal{U},\widetilde{\Omega}_{\P_\Sigma}^{n+1}).
$$
In order to write this element relative to the covering $\mathcal{U}^I$ we can use any non-trivial Euler relation
$$
1=\frac{1}{k_\ell F}\left(x_{\ell}\frac{\partial F}{\partial x_{\ell}}-\sum_{i=1}^{n+1}a_{ i,\ell}x_{i}\frac{\partial F}{\partial x_{i}}\right).
$$
Using this relation and the fact that $\det(\rho_{\{1,\ldots,n+1\}\setminus\{i\}\cup\{\ell\}})=(-1)^{i+1}a_{i,\ell}\det(\rho_1|\cdots|\rho_{n+1})$ we conclude that
$$
\delta(\omega_P\cup\omega_Q)=\frac{(-1)^\frac{n}{2}PQ\Omega}{(\frac{n}{2}!)^2F\cdot k_\ell F}\left(\frac{\widehat{x_{1,\ldots,{n+1}}}}{F_{1,\ldots,{n+1}}}\det(\rho_1|\cdots|\rho_{n+1})\left(x_{\ell}F_{\ell}-\sum_{i=1}^{n+1}a_{i,\ell}x_{i}F_{i}\right)\right)
$$
$$
=\frac{(-1)^{\frac{n}{2}}PQ\widehat{x_{1,\ldots,{n+1}}}\det(\rho_1|\cdots|\rho_{n+1})}{(\frac{n}{2}!)^2F\cdot F_{1}\cdots F_{{n+1}}}\Omega\in H^{n+1}(\U^I,\widetilde{\Omega}_{\P_\Sigma}^{n+1}).
$$
\end{proof}

\noindent\textbf{Proof of \cref{thm4}} Since the coboundary map \eqref{coboundary} preserves the trace, it follows by \cref{propcup} and \cref{cortrace} that
\begin{equation}
\label{eqcup}
\text{Tr}(\omega_P\cup\omega_Q)=\text{Tr}(\delta(\omega_P\cup\omega_Q))=(-1)^{\frac{n}{2}+1}\cdot c'\cdot (n+1)!\text{Vol}(\Delta)
\end{equation}
where $\Delta$ is the convex polyhedron associated to $\beta$ and $c'\in\C$ is the unique number such that
\begin{equation}
\label{eqcupfin}
\frac{(-1)^\frac{n}{2}PQx_1\cdots x_r\det(\rho_I)}{(\frac{n}{2}!)^2}\equiv c'\cdot \Hess^I_\Sigma(F)\hspace{5mm}\text{(mod }J_0(F))    
\end{equation}
and the result follows. \hfill $\blacksquare$

\begin{cor}
\label{corhess}
Let $\P_\Sigma$ be a projective simplicial toric variety with anti-canonical class $\beta_0\in\Cl(\P_\Sigma)$. Let $X=\{F=0\}\subseteq\P_\Sigma$ be an ample quasi-smooth hypersurface of degree $\beta=\deg(F)\in \Cl(\P_\Sigma)$ and even dimension $n$. If $H^{\frac{n}{2},\frac{n}{2}}(X)_\prim\neq 0$, then $\Hess_\Sigma(F)\in R_0(F)^{(n+2)\beta-\beta_0}$ is well-defined.
\end{cor}

\begin{proof}
Take any $\omega_P\in H^{\frac{n}{2},\frac{n}{2}}(X)_\prim\setminus\{0\}$, then since the cup product is unimodular, there exists some $\omega_Q\in H^{\frac{n}{2},\frac{n}{2}}(X)_\prim$ such that $\omega_P\cup\omega_Q\neq 0$. Following the proof of \cref{thm4} we can compute the its trace, which is non-zero, with respect to any set $\rho_{i_1}, \ldots,\rho_{i_{n+1}}$ of linearly independent primitive generators of rays of $\Sigma$. Since this value does not depend on this choice, it follows from \eqref{eqcup} and \eqref{eqcupfin} that 
$$
\frac{\Hess^I_\Sigma(F)}{\det(\rho_I)}=\frac{-PQx_1\cdots x_r(n+1)!\text{Vol}(\Delta)}{\text{Tr}(\omega_P\cup\omega_Q)\cdot (\frac{n}{2}!)^2}\in R_0(F)^{(n+2)\beta-\beta_0}
$$
does not depend of $I$. 
\end{proof}

Now that we have all the necessary ingredients to prove the main theorems.

\bigskip

\noindent\textbf{Proof of \cref{thm2}} Let us consider first the case where $X=\{F=0\}$ is an hypersurface. Since the trace of $\omega_P|_Z$ is constant along flat deformations of the pair $(Z\subseteq X)$, we can assume that we are in the same context of \cref{propPoincres}. Now, since the coboundary map preserves the trace we have
$$
\text{Tr}(\omega_P|_Z)=\text{Tr}(\delta^{\frac{n}{2}+1}(\omega_P|_Z)).
$$
The result follows once we compare \cref{coriteratedres} with \cref{cortrace} and we use the fact that by the Euler relations $k_{j,\ell}e_j=x_{\ell}\frac{\partial e_j}{\partial x_{\ell}}-\sum_{k=1}^{n+1}a_{i_k,\ell}x_{i_k}\frac{\partial e_j}{\partial x_{i_k}}$ we have
$$
\det\begin{pmatrix}
k_{0,\ell}e_0 & \cdots & k_{n+1, \ell}e_{n+1} \\ \frac{\partial e_0}{\partial x_{i_1}} & \cdots & \frac{\partial e_{n+1}}{\partial x_{i_1}} \\ \vdots &  & \vdots \\ \frac{\partial e_0}{\partial x_{i_{n+1}}} & \cdots & \frac{\partial e_{n+1}}{\partial x_{i_{n+1}}} 
\end{pmatrix}=x_{\ell}\det\begin{pmatrix}
\frac{\partial e_0}{\partial x_{\ell}} & \cdots & \frac{\partial e_{n+1}}{\partial x_{\ell}} \\ \frac{\partial e_0}{\partial x_{i_1}} & \cdots & \frac{\partial e_{n+1}}{\partial x_{i_1}} \\ \vdots &  & \vdots \\ \frac{\partial e_0}{\partial x_{i_{n+1}}} & \cdots & \frac{\partial e_{n+1}}{\partial x_{i_{n+1}}} 
\end{pmatrix}.
$$
For the case where $X$ is a non-degenerated intersection, we use the Cayley trick \cref{rmkCayleytrick} and the fact that $\text{Tr}(\omega_P|_Z)=\text{Tr}(\widetilde\omega_P|_{\pi^{-1}(Z)})$ (which follows from the factorization \eqref{eqCaytr}) where $\widetilde\omega_P\in H^{\frac{m+s-2}{2},\frac{m+s-2}{2}}(Y)_\prim$ is the residue form induced by $P\in S(\widetilde{\Sigma})$.
\hfill $\blacksquare$

\bigskip

\noindent\textbf{Proof of \cref{mainthm}} By \cref{thm2} and \cref{thm4} it follows that 
$$
\text{Tr}(\widetilde\omega_P|_{\pi^{-1}(Z)})=\text{Tr}(\widetilde\omega_P\cup \widetilde\omega_{P_Z})
$$
for all $P\in S(\widetilde\Sigma)^{(\frac{n}{2}+1)\beta-\beta_0}$. By Poincaré duality this implies that $[\pi^{-1}(Z)]_\prim=\widetilde\omega_{P_Z}$, and so $[Z]_\prim=\omega_{P_Z}$. In consequence $[Z]=[\alpha]|_X+\omega_{P_Z}$ for some $\alpha\in \CH^\frac{n}{2}(\P_\Sigma)_\Q$. In order to compute it we see that for any $\nu\in H^s(\P_\Sigma,\C)$
$$
\text{Tr}([Z]\cup\nu)=\text{Tr}(\nu|_Z)=\text{Tr}([\alpha]|_X\cup\nu|_X)=\text{Tr}([X]\cup[\alpha]\cup\nu)
$$
and so $[Z]=[X]\cup[\alpha]$, which is equivalent to say that $\alpha_1\cdots\alpha_\frac{m+s}{2}=\beta_1\cdots\beta_s\cdot\alpha$. This relation characterizes $\alpha$ by the toric Hard Lefschetz theorem.
\hfill $\blacksquare$







\bibliographystyle{alpha}

\bibliography{ref}



\end{document}